\documentclass[a4paper,doublespace]{amsart}

\usepackage[colorlinks, linkcolor=red, anchorcolor=blue, citecolor=green]{hyperref}
\usepackage{graphics,epic}
\usepackage{amsmath,amssymb, amsthm}
\usepackage[all,2cell]{xy}
\usepackage{tikz}
\usetikzlibrary{arrows}
\usepackage{shorttoc}

\newtheorem{theorem}{Theorem}[section]
\newtheorem*{theorem*}{Theorem}

\newtheorem{lemma}[theorem]{Lemma}
\newtheorem{proposition}[theorem]{Proposition}
\newtheorem{corollary}[theorem]{Corollary}

\newtheorem*{conjecture*}{Conjecture}

\newtheorem{example}[theorem]{Example}
\newtheorem{remark}[theorem]{Remark}

\newcommand{\opname}[1]{\operatorname{\mathsf{#1}}}

\renewcommand{\mod}{\opname{mod}\nolimits}

\newcommand{\add}{\opname{add}\nolimits}

\newcommand{\der}{\cd}

\newcommand{\dimv}{\underline{\dim}\,}

\newcommand{\rank}{\opname{rank}\nolimits}
\newcommand{\rankv}{\underline{\rank}\,}

\newcommand{\coind}{\opname{coind}}

\newcommand{\rep}{\opname{rep}\nolimits}

\newcommand{\im}{\opname{im}\nolimits}

\newcommand{\Z}{\mathbb{Z}}
\newcommand{\N}{\mathbb{N}}

\newcommand{\C}{\mathbb{C}}

\renewcommand{\P}{\mathbb{P}}

\newcommand{\sgn}{\opname{sgn}}

\newcommand{\Hom}{\opname{Hom}}
\newcommand{\go}{\opname{G_0}}

\newcommand{\Ext}{\opname{Ext}}

\newcommand{\End}{\opname{End}}

\newcommand{\ca}{{\mathcal A}}

\newcommand{\cc}{{\mathcal C}}
\newcommand{\cd}{{\mathcal D}}

\newcommand{\cs}{{\mathcal S}}
\newcommand{\ct}{{\mathcal T}}

\newcommand{\cw}{{\mathcal W}}

\renewcommand{\hat}[1]{\widehat{#1}}

\begin{document}

\title[On cluster algebra of type $\mathrm{C}$]{Cluster algebras arising from cluster tubes II: the Caldero-Chapoton map}\thanks{Partially  supported by the National Natural Science Foundation of China (No. 11471224)}

\author{Changjian Fu}
\address{Changjian Fu\\Department of Mathematics\\SiChuan University\\610064 Chengdu\\P.R.China}
\email{changjianfu@scu.edu.cn}
\author{Shengfei Geng}
\address{Shengfei Geng\\Department of Mathematics\\SiChuan University\\610064 Chengdu\\P.R.China}
\email{genshengfei@scu.edu.cn}
\author{Pin Liu}
\address{Pin Liu\\Department of Mathematics\\
   Southwest Jiaotong University\\
  610031 Chengdu \\
  P.R.China}
  \email{pinliu@swjtu.edu.cn}
\maketitle

\begin{abstract}
We continue our investigation on cluster algebras arising from cluster tubes. Let $\cc$ be a cluster tube of rank $n+1$. For an arbitrary basic maximal rigid object $T$ of $\cc$, one may associate a skew-symmetrizable integer matrix $B_T$  and  hence a cluster algebra $\mathcal{A}(B_T)$ to $T$.
 We define an analogue Caldero-Chapoton map $\mathbb{X}_M^T$ for each indecomposable rigid object $M\in \cc$ and prove that $\mathbb{X}_?^T$ yields a bijection between the indecomposable rigid objects of $\cc$ and the cluster variables of the cluster algebra $\mathcal{A}(B_T)$.  The construction of the Caldero-Chapoton map involves Grassmannians of locally free submodules over the endomorphism algebra of $T$. We also show that there is a non-trivial $\mathbb{C}^{\times}$-action on the Grassmannians of locally free submodules, which is of independent interest.
\end{abstract}

\tableofcontents

\section{Introduction}~\label{s:introduction}
Cluster algebras were introduced by Fomin and Zelevinsky~\cite{FominZelevinsky02} with the aim to set up a combinatorial framework for the study of total positivity in algebraic groups and canonical bases in quantum groups.
Since their appearance, the combinatorial structure of cluster algebras has rapidly been found  in many other branches of mathematics such as Poisson geometry, Teichm\"{u}ller theory and representation theory of algebras~({\it cf.}~\cite{Keller10}). The discovery  in representation theory motivated  the development of cluster-tilting theory in cluster categories~\cite{BMRRT} or more generally in (2-Calabi-Yau) triangulated categories~\cite{IY,KR,KZ}.  Meanwhile, cluster-tilting theory provided us with  additive categorifications for cluster algebras of skew-symmetric type. In the categorification theory, the cluster-tilting objects and the so-called Caldero-Chapoton map~\cite{CalderoChapoton} or Palu's cluster character~\cite{Palu} played central roles, which enable us to realize cluster algebra via the geometric structures of objects in the corresponding $2$-Calabi-Yau category ({\it cf.} also~\cite{GLS08, Plamondon11}).  

Cluster tubes are $2$-Calabi-Yau triangulated categories, which admit no cluster-tilting objects but only maximal rigid objects. Nevertheless, Buan {\it et al.}~\cite{BuanMarshVatne} proved that the combinatorics of basic maximal rigid objects in cluster tubes provided categorifications for the combinatorics of cluster algebras of type $\mathrm{C}$. In~\cite{ZhouZhu14}, an analogue Caldero-Chapoton map has been established with respect to a special choice of acyclic initial seed, which implies that cluster tubes do provide additive categorifications for cluster algebras of type $\mathrm{C}$. However, we remark that most of the combinatorics and properties ({\it e.g.} the denominator vectors, $g$-vectors and the Caldero-Chapoton map) of a cluster algebra depend on the choice of the initial seed.

The present paper continues our investigation on cluster algebras arising from cluster tubes in a full generality initiated in~\cite{FGL1}. Here we establish the Caldero-Chapoton formula $\mathbb{X}^T_?$ with respect to an arbitrary initial seed ({\it cf.}~Section~\ref{s:index-coindex} and Theorem~\ref{t:main-thm-restate}), where $T$ is a basic maximal rigid object.
This generalizes the work of Zhou and Zhu~\cite{ZhouZhu14} in a full generality and completes the existence of Caldero-Chapoton maps for cluster algebras of type $\mathrm{C}$.
Our definition of Caldero-Chapoton map $\mathbb{X}^T_?$ involves varieties consisting of locally free submodules over the endomorphism algebra of $T$.
We remark that the consideration of variety consisting of locally free submodules also appears in the recent work of Geiss {\it et al.}~\cite{GLS17}. In~\cite{GLS17}, Geiss {\it et al.} obtained a Caldero-Chapoton formula for an arbitrary cluster algebra of finite type with acyclic initial seed via the category of locally free modules over certain Iwanaga-Gorenstein algebras introduced in~\cite{GLS14}.
Compare to~\cite{GLS17}, we only deal with cluster algebras of type $\mathrm{C}$ but with no restriction on initial seeds. Even for the case of cluster algebras of type $\mathrm{C}$ with acyclic initial seeds, our proof of Theorem~\ref{t:main-thm-restate} is also completely different from the one of~\cite{GLS17}. In fact, our proof goes back to ~\cite{CalderoChapoton,DG}. However, it is not possible to apply the proof of~\cite{CalderoChapoton,DG} directly, since in general the intersection of locally free submodules is no longer locally free.  In order to establish the main result, we show that there is a natural $\mathbb{C}^{\times}$-action on the Grassmannians of locally free submodules.
Another key step in our proof is to introduce an auxiliary Grassmannian which considers locally free submodules with an extra structure ~({\it cf.}~Section~\ref{s:grassmanian} for details). 
We remark that there are other generalization of the Caldero-Chapoton formula, {\it cf.} for example \cite{Demonet, R11,R15,GLS17} for skew-symmetrizable type and \cite{CHL17} for sign-skew-symmetric type.

The paper is organized as follows. In Section~\ref{s:preliminaries}, we recall the required background from cluster algebras and cluster tubes. 
After introducing the necessary definitions and notations, we state the main result (Theorem~\ref{t:main-thm}) in Section~\ref{s:main-thm}.  We restate the main result as Thoerem~\ref{t:main-thm-restate} in Section~\ref{s:index-coindex} by using coindex.
The rest of Section~\ref{s:index-coindex} is then devoted to study the properties of index and coindex.
 In Section~\ref{s:grassmanian}, we show that there is a natural $\mathbb{C}^{\times}$-action on Grassmannians of locally free submodules (Lemma~\ref{l:c-action}). We introduce an auxiliary Grassmannian with respect to an Auslander-Reiten sequence and study the basic properties for Grassmannians of locally free submodules following~\cite{CalderoChapoton,DG}.  The results obtained in Section~\ref{s:index-coindex} and~\ref{s:grassmanian} are applied to prove our main result (Theorem~\ref{t:main-thm}) in Section~\ref{s:the-proof}. 

\noindent{\bf Convention.}
Let $\mathbb{C}$ be the field of complex number. We denote by $\mathbb{C}^{\times}$ the set of nonzero elements in $\mathbb{C}$. Fix a positive integer $n$, we denote by $\underline{e_1},\cdots, \underline{e_n}$ the standard basis of $\Z^n$. For a matrix $B$, we denote by $B^{\opname{tr}}$ the transpose of $B$.
For a vector $\beta\in \Z^n$ and $1\leq i\leq n$, let $[\beta:i]$ be the $i$-th component of $\beta$. For an object $M$ in a category $\cc$, denote by $M^{\oplus n}$ the direct sum of $n$ copies of $M$ and $|M|$ the number of non-isomorphic indecomposable direct summands of $M$. Denote by $\add M$ the subcategory of $\cc$ consisting of objects which are finite direct sum of direct summands of $M$. 
For an arrow $\alpha:i\to j$, we denote by $s(\alpha)=i$ the source of $\alpha$ and $t(\alpha)=j$ the target of $\alpha$.

\section{Cluster algebras and cluster tubes}~\label{s:preliminaries}
Throughout this article, we fix a positive integer $n$.
\subsection{Cluster algebras}
For an integer $a$, we set $[a]_+=\max\{a, 0\}$ and $\sgn(a)=\begin{cases}\frac{a}{|a|}& a\neq 0\\ 0& a=0\end{cases}.$  Let $\mathcal{F}$ be the field of rational functions in $n$ indeterminates with coefficients in $\mathbb{Q}$.
A {\it seed} of $\mathcal{F}$ is a pair $(B, \mathbf{x})$ consisting of a skew-symmetrizable matrix $B=(b_{ij})\in M_n(\Z)$ and a free generating set $\mathbf{x}=\{x_1,\cdots, x_n\}$ of the field $\mathcal{F}$. 

 For any $1\leq k\leq n$, the {\it seed mutation of $(B,\mathbf{x})$ in the direction $k$} transforms $(B,\mathbf{x})$ into a new seed $\mu_k(B, \mathbf{x})=(B',\mathbf{x}')$, where
 \begin{itemize}
 \item[$\bullet$] the entries $b_{ij}'$ of $B'$ are given by
\[b_{ij}'=\begin{cases}-b_{ij} &\text{if }i=k ~\text{or}~j=k,\\
b_{ij}+\opname{sgn}(b_{ik})[b_{ik}b_{kj}]_{+}& \text{else}.\end{cases}
\]
\item[$\bullet$] the cluster $\mathbf{x}'=\{x_1',\cdots, x_n'\}$ is given by $x_j'=x_j$ for $j\neq k$ and $x_k'\in \mathcal{F}$ is determined by the {\it exchange relation}
\[x_k'x_k=\prod_{i=1}^nx_i^{[b_{ik}]_+}+\prod_{i=1}^nx_i^{[-b_{ik}]_+}.
\]
\end{itemize}

 Let $\mathbb{T}_n$ be the $n$-regular tree whose edges are labeled by the numbers $1,2,\cdots, n$, so that the $n$ edges emanating from each vertex receive different labels.
Fix a seed $(B,\mathbf{x})$ as initial seed, a {\it cluster pattern} is the assignment of a seed $(B_t,\mathbf{x}_t)$ to each vertex $t$ of $\mathbb{T}_n$ such that 
\begin{itemize}
\item the seeds assigned to vertices $t$ and $t'$ linked by an edge labeled $k$ are obtained from each other by  the seed mutation $\mu_k$;
\item there exists a vertex $t_0\in \mathbb{T}_n$ such that $(B_{t_0},\mathbf{x}_{t_0})=(B,\mathrm{x})$.
\end{itemize}
A cluster pattern is uniquely determined by an  assignment of the initial seed $(B, \mathbf{x})$ to any vertex $t_0\in \mathbb{T}_n$.
The matrix $B_t$ is the {\it exchange matrix} and $\mathbf{x}_t$ is the {\it cluster} of the seed $(B_t,\mathbf{x}_t)$. Elements of the cluster $\mathbf{x}_t$ are {\it cluster variables} of $(B_t,\mathbf{x}_t)$. Cluster variables in $\mathbf{x}_{t_0}$ are {\it initial cluster variables}. 
The {\it cluster algebra} $\mathcal{A}(B):=\mathcal{A}(B,\mathbf{x})$ is the subalgebra of $\mathcal{F}$ generated by all of the cluster variables.
The cluster algebra $\mathcal{A}(B)$ is of {\it skew-symmetric type} provided that $B$ is skew-symmetric, otherwise  $\mathcal{A}(B)$ is of {\it skew-symmetrizable type}.

A cluster algebra is of {\it finite type} if it has only finitely many cluster variables. Fomin and Zelevinsky~\cite{FominZelevinsky03} proved that the cluster algebras of finite type are parametrized by the finite root systems. Let $B=(b_{ij})\in M_n(\Z)$ be a skew-symmetrizable matrix. The Cartan couterpart $A(B)=(c_{ij})$ of $B$ is an $n\times n$ integer matrix such that $c_{ii}=2$ for $i=1,\cdots, n$ and $c_{ij}=-|b_{ij}|$ for $i\neq j$. In particular, $A(B)$ is a generalized Cartan matrix. According to~\cite{FominZelevinsky03}, the cluster algebra $\mathcal{A}(B)$ is {\it of type $\mathrm{C}_n$} if there is a vertex $t\in \mathbb{T}_n$ such that the Cartan counterpart $A(B_t)$ is a generalized Cartan matrix of type $\mathrm{C}_n$.

\subsection{Cluster tubes}

 Let $\Delta_{n+1}$ be the cyclic quiver with $n+1$ vertices. We label the vertex set by $\{1,2,\cdots, n+1\}$ such that the arrows are precisely from vertex $i$ to $i+1$ (taken modulo $n+1$). Denote by $\ct:=\ct_{n+1}$ the category of finite-dimensional nilpotent $\C$-representations over the opposite quiver $\Delta_{n+1}^{\operatorname{op}}$.  The category $\ct$ is called a {\it tube of rank $n+1$}, which is a hereditary abelian category. Moreover, $\ct$ is standard, {\it i.e.} the subcategory of $\ct$ consisting of the indecomposable objects is equivalent to the mesh category of the Auslander-Reiten (AR for short) quiver of $\ct$.
 Each indecomposable object of $\ct$ is uniquely determined by its socle and its length. For $1\leq a\leq n+1$ and $b\in \N$, we will denote by $(a,b)$ the unique indecomposable object with socle the simple at vertex $a$ and of length $b$. 
 
 Let $\der^b(\ct)$ be the bounded derived category of $\ct$ with suspension functor $\Sigma$. Let $\tau$ be the AR translation of $\der^b(\ct)$, where $\tau(a, b)=(a-1, b)$. The {\it cluster tube of rank $n+1$} is the orbit category $\cc:=\cc_{n+1}=\der^b(\ct)/ \tau^{-1}\circ \Sigma$. The category $\cc$ admits a canonical triangle structure such that the canonical projection $\pi:\der^b(\ct)\to \cc$ is a triangle functor~\cite{Keller05}({\it cf.}~also \cite{Barot-Kussin-Lenzing}). Moreover, it is a Calabi-Yau triangulated category with Calabi-Yau dimension of $2$ ({\it cf.}~\cite{Keller05}).
 The AR-translation $\tau$ and the suspension functor $\Sigma$ of $\der^b(\ct) $ induce the AR-translation and the suspension functor of $\cc$. By abuse of notations, we still denote by $\tau$ and $\Sigma$ the AR-translation and the suspension functor of $\cc$ respectively. We remark that there is an isomorphism of functors $\tau\xrightarrow{\sim}\Sigma$ in the category $\cc$.
The composition of the embedding of $\ct$ into $\der^b(\ct)$ with the canonical projection $\pi$ yields a bijection between the indecomposable objects of $\ct$ and the indecomposable objects of $\cc$. We always identify the objects of  $\ct$ with the ones of $\cc$ by the bijection. In particular, we may say the length of an indecomposable object of $\cc$. For an indecomposable object $X\in \cc$, we will denote by $l(X)$ the length of $X$.

Let $X, Y$ be indecomposable objects of $\ct$, by definition of $\cc$ and the fact that $\ct$ is hereditary, we have
\[\Hom_{\cc}(X, Y)=\Hom_{\der^b(\ct)}(X, Y)\oplus \Hom_{\der^b(\ct)}(X, \tau^{-1}\circ \Sigma Y).
\]
According to the Auslander-Reiten duality, we obtain ({\it cf.} Lemma $2.1$ of ~\cite{BuanMarshVatne})
\begin{lemma}~\label{l:compute-morphism-cluster-tube}
Let $X, Y$ be indecomposable objects of $\ct$,  we have
\[\Hom_{\cc}(X, Y)\cong \Hom_{\ct}(X, Y)\oplus D\Hom_{\ct}(Y, \tau^2 X).
\]
\end{lemma}

Following~\cite{BuanMarshVatne}, morphisms in $\Hom_{\der^b(\ct)}(X, Y)$ are called {\it $\ct$-morphisms} from $X$ to $Y$ and morphisms in $\Hom_{\der^b(\ct)}(X, \tau^{-1}\circ \Sigma Y)$ are called {\it $\der$-morphisms} from $X$ to $Y$. Each morphism from $X$ to $Y$ in $\cc$ can be written as the sum of a $\ct$-morphism with a $\der$-morphism. It is also well-known that the composition of two $\ct$-morphisms is also a $\ct$-morphism, the composition of a $\ct$-morphism with a $\der$-morphism is a $\der$-morphism, and the composition of two $\der$-morphisms is zero, and no $\ct$-morphism can factor through a $\der$-morphism.

For each indecomposable object $X=(a, b)\in \ct$ with $b\leq n$, the {\it wing $\cw_{X}$ determined by $X$} is the set of indecomposables whose position in the AR-quiver is in the triangle with $X$ on top.
We also denote by $X^{\sqsubset}$ the support of $\Hom_{\ct}(X,-)$ in $\ct$. Namely, $X^{\sqsubset}$ consists of indecomposable objects $Y$ of $\ct$ such that $\Hom_{\der^b(\ct)}(X,Y)\neq 0$. Dually, we may define $~^{\sqsupset}X$ to be the support of $\Hom_{\ct}(-,X)$ in $\ct$. By Lemma~\ref{l:compute-morphism-cluster-tube}, we clearly know  that an indecomposable object $Y\in \cc$ satisfies that $\Hom_{\cc}(X,Y)\neq 0$ if and only if $Y\in X^{\sqsubset}\cup~^{\sqsupset}\tau^2X$.

The following result is a consequence of Lemma~\ref{l:compute-morphism-cluster-tube} and the fact that $\ct$ is standard.
\begin{lemma}~\label{l:morphism-non-vanish}
Let $X$ and $Y$ be indecomposable objects of $\cc$ with $l(X)=n$. If $Y\in \cw_{\tau X}$, then $\Hom_\cc(X, Y)=0$; if $Y\in X^\sqsubset\cap^\sqsupset\tau^2X$, then $\dim_\C\Hom_\cc(X, Y)=2$.
\end{lemma}

\subsection{Algebras arising from cluster tubes}~\label{ss: endomorphism-quiver}
We first recall the definition of basic maximal rigid objects in a $2$-Calabi-Yau triangulated category.
Let $\der$ be a $\Hom$-finite $2$-Calabi-Yau triangulated category with suspension functor $\Sigma$. An object $T\in \der$ is {\it rigid} provided $\Ext^1_{\der}(T,  T)=0$ and it is {\it maximal rigid}  if it is rigid and $\Ext^1_{\der}(X\oplus T, X\oplus
T)=0$ implies that $X\in \add T$. A rigid object $T\in \der$ is called a {\it cluster-tilting object} if $\Ext^1_{\der}(T, Y)=0$ implies that $Y\in \add T$.

 It was proved in ~\cite{BuanMarshVatne} that the cluster tube $\cc$ has no cluster-tilting objects but only maximal rigid objects. Moreover, the following descriptions of maximal rigid objects were given.
\begin{lemma}~\label{l:basic-property-cluster-tube}
\begin{itemize}
\item[(1)] An indecomposable object $X$ in $\cc$ is rigid if and only if it has length $l(X)\leq n$;
\item[(2)] Let $T$ be an arbitrary basic maximal rigid object of $\cc$. Then $|T|=n$ and $T$ admits a unique indecomposable direct summand, say $T_1$, with length $n$. Moreover, each indecomposable direct summand of $T$ lies in  the wing $\cw_{T_1}$.
\end{itemize}
\end{lemma}

Let $T=\bigoplus_{i=1}^nT_i$ be a  basic maximal rigid object of $\cc$ with indecomposable direct summands $T_1,\cdots, T_n$. 
Denote by $A:=\End_{\cc}(T)$ the endomorphism algebra of $T$ and $\mod A$ the category of finitely generated right $A$-modules. An object $X\in \cc$ is {\it finitely presented} by $T$ if there is a triangle
\[T_1^X\to T_0^X\xrightarrow{f} X\to \Sigma T_1^X
\]
in $\cc$, where $T_1^X, T_0^X\in \add T$.
Let $\opname{pr}T$ be the  full subcategory of $\cc$ consisting of objects which are finitely presented by $T$.
A general result of ~\cite{ZhouZhu} implies that rigid objects of $\cc$ belong to $\opname{pr}T$.
 The functor \[F:=\Hom_{\cc}(T,-):\cc\to \mod A\]
induces an equivalence of categories \[F: \opname{pr}T/\add \Sigma T\xrightarrow{\sim} \mod A,\] where $\opname{pr}T/\add \Sigma T$ is the additive quotient of $\opname{pr} T$ by morphisms factorizing through $\add \Sigma T$. Moreover, the restriction of the functor $F$ to the subcategory $\add T$ ({\it resp.} $\add \Sigma^2T$) yields an equivalence between $\add T$ ({\it resp.} $\add \Sigma^2T$) and the category of finitely generated projective ({\it resp.} injective) $A$-modules. 
Let $\tau_A$ be the AR translation of $\mod A$. Recall that an $A$-module $M$ is {\it $\tau_A$-rigid} if $\Hom_A(M,\tau_A M)=0$ ({\it cf.}~\cite{AIR}). The following bijection has been proved in~\cite{ChangZhangZhu, LiuXie}.
\begin{proposition}~\label{p:tau-rigid}
The functor $F$ induces a bijection between the indecomposable rigid objects in $\cc\backslash \add\Sigma T$ and the indecomposable $\tau_A$-rigid $A$-modules.
\end{proposition}

We now associate a quiver $Q_T$ to $T$, whose vertices correspond to the indecomposable direct summands of $T$ and the arrows from the indecomposable direct summand $T_i$ to $T_j$ is given by the dimension of the space of irreducible maps $\opname{rad}(T_i,T_j)/\opname{rad}^2(T_i,T_j)$, where $\opname{rad}(-,-)$ is the radical of the category $\add T$. For simplicity, we will also denote the vertex set $Q_0$ of $Q_T$ by $Q_0=\{1,2,\cdots, n\}$, where the vertex $i$ corresponds to $T_i$ for $i=1,\cdots, n$. It is clear that the quiver $Q_T$ coincides with the Gabriel quiver of $A$.

Let $\mathcal{Q}_{n}$ be the set of quivers with $n$ vertices satisfying the following conditions:
\begin{itemize}
\item[(a)] All non-trivial minimal cycles of length at least $2$ in the underlying graph is oriented and of length $3$;
\item[(b)] Any vertex has at most four neighbors;
\item[(b)] If a vertex has four neighbors, then two of its adjacent arrows belong to one $3$-cycle, and the other two belong to another $3$-cycle;
    \item[(d)]If a vertex has three neighbors, then two of its adjacent arrows belong to one $3$-cycle, and the third one does not belong to any $3$-cycle;
        \item[(e)] There is a unique loop $\rho$ at a vertex $t$ which has one neighbor, or has two neighbors and its traversed by a $3$-cycle.
\end{itemize}
It has been proved in~\cite{Vatne11, Yang12} that the quiver $Q_T$ belongs to $\mathcal{Q}_n$. Moreover,  the endomorphism algebra $A$ of $T$ is determined by its underlying Gabriel quiver $Q_T$. Namely,
\begin{theorem}~\label{t:algebra-maximal-tube}
 An algebra is the endomorphism algebra of a basic maximal rigid object in the cluster tube $\cc$ if and only if it is isomorphic to $\C Q/I$ for some $Q\in \mathcal{Q}_{n}$, where $I$ is the ideal generated by the square of the unique loop $\rho$ and all paths of length $2$ in a $3$-cycle.
\end{theorem}
As a direct consequence of Theorem~\ref{t:algebra-maximal-tube},  the endomorphism algebra $A=\End_\cc(T)$ is a gentle algebra and there is  a unique loop $\rho$ in the Gabriel quiver $Q_T$ of $A$. 
 It is clear that the source of the unique loop $\rho$ is the vertex corresponding to the unique indecomposable direct summand of $T$ with length $n$. We refer to~\cite{ButlerRingel} for the definition and basic results for gentle algebras. For later use, let us recall the definition of string and the construction of string modules for the algebra $A$. Let $I_T$ be the ideal of $\C Q_T$ as in Theorem~\ref{t:algebra-maximal-tube}.
 
 For any arrow $\beta$ of $Q_T$, denote by $\beta^{-1}$ a formal inverse for $\beta$, with $s(\beta^{-1}):=t(\beta)$ and $t(\beta^{-1}):=s(\beta)$. A word $w=\alpha_n\alpha_{n-1}\cdots \alpha_1$ of arrows and their formal inverses is called a {\it string} of length $n\geq 1$ if $\alpha_{i}\neq \alpha_{i+1}^{-1}, t(\alpha_i)=s(\alpha_{i+1})$ for all $1\leq i<n$, and no subword $\alpha_{i+r}\alpha_{i+r-1}\cdots\alpha_{i}$ or its inverse belongs to $I_T$. For each vertex $u$ of $Q_T$, we define two strings $1_{(u,\delta)}, \delta=\pm 1$, with both $s(1_{(u,\delta)})=u$ and $t(1_{(u,\delta)})=u$, and define $1^{-1}_{(u,\delta)}=1_{(u,-\delta)}$.  For each string $w$ of $A$, one defines a unique {\it string module} $M(w)$ for $A$ following~\cite{ButlerRingel}. Note that a right $A$-module identifies a representation of $Q_T^{op}$ satisfying the opposite relations $I_T^{op}$.
For any string $w=\alpha_n\alpha_{n-1}\cdots \alpha_1$ or $w=1_{(u,t)}$ of $A$, we define a representation $M(w)$ of $Q_T^{op}$ as follows. Let $u(i)=s(\alpha_{i+1}), 0\leq i<n$ and $u(n)=t(\alpha_n)=t(w)$. For each vertex $v$ of $Q_T^{op}$, set $I_v=\{i~|~u(i)=v\}\subseteq \{0,1,\cdots, n\}$. Let $M(w)_v$ be a vector space spanned by the base vectors $z_i,i\in I_v$. For any arrow $\beta$ in $Q_T$, if $\alpha_i=\beta$, define $\beta^{-1}(z_i)=z_{i-1}$; if $\alpha_i=\beta^{-1}$, define $\beta^{-1}(z_{i-1})=z_i$ for $1\leq i\leq n$; otherwise set the linear map associated to $\beta^{-1}$ to be zero. Then $M(w)$ is the right string $A$-module associated to the string $w$.
 
  \begin{example}~\label{ex:string-module}
 Let $A$ be the $\C$-algebra given by the quiver $Q$
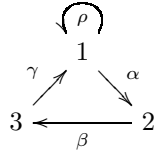
\[\xymatrix@=12pt{&1\ar@{->}@(ur,ul)^{\rho}\ar[dr]^{\alpha}\\
3\ar[ur]^{\gamma}&&2\ar[ll]^{\beta}}
\]bound by the relations $\beta\alpha=0, \gamma\beta=0, \alpha\gamma=0$ and $\rho^2=0$. We know that $A$ is the endomorphism algebra of a basic maximal rigid object, say $T=T_1\oplus T_2\oplus T_3$,  in the cluster tube $\cc_{4}$ by Theorem~\ref{t:algebra-maximal-tube}.

Let $w_1=\alpha\rho\gamma$ and $w_2=\alpha\rho^{-1}\gamma$ be two strings for $A$. The right $A$-modules $M(w_1)$ and $M(w_2)$ are described as follows
\[M(w_1)=\xymatrix@=12pt{&\C\oplus \C\ar@{->}@(ul,ur)^{\tiny \begin{bmatrix}0&1\\0&0\end{bmatrix}}\ar[dl]_{\tiny \begin{bmatrix}1\\ 0\end{bmatrix}}\\
\C\ar[rr]^{0}&&\C\ar[ul]_{\tiny \begin{bmatrix}0&1\end{bmatrix}}}~~~\quad
 M(w_2)=\xymatrix@=12pt{&\C\oplus \C\ar@{->}@(ul,ur)^{\tiny\begin{bmatrix}0&0\\1&0\end{bmatrix}}\ar[dl]_{\tiny\begin{bmatrix}1\\ 0\end{bmatrix}}\\
\C\ar[rr]^{0}&&\C.\ar[ul]_{\tiny\begin{bmatrix}0&1\end{bmatrix}}}
\]
 \end{example}
 
 It is easy to observe that $A$ is of representation-finite type (cf. ~\cite{Vatne11, Yang12}). Hence each indecomposable $A$-module is a string module and has a presentation as $M(w)$ for some string $w$. 
\begin{lemma}~\label{l:dimension-vector-loop}
Let $T=\bigoplus_{i=1}^nT_i$ be a basic maximal rigid object of $\cc$ with $l(T_1)=n$. Denote by $A=\End_\cc(T)$ the endomorphism algebra of $T$.
For each indecomposable $A$-module $M$, we have either
$[\dimv M: 1]=0 $~ or~$[\dimv M:1]=2$. 
\end{lemma}
\begin{proof}
By the assumption $l(T_1)=n$, we know that the source of the unique loop $\rho$ is the vertex $1$. According to Theorem~\ref{t:algebra-maximal-tube} and the description of $\mathcal{Q}_n$, we conclude that for each string $w$, $\rho$ and its formal inverse $\rho^{-1}$ appear in $w$ at most once. Now the result follows from the construction of string modules.
\end{proof}

\subsection{Cluster algebras arising from cluster tube}~\label{s: skew-symmetrizable-matrix}
For a given basic maximal rigid object $T=\overline{T}\oplus T_k$ in $\cc$ with indecomposable direct summand $T_k$,  the {\it mutation $\mu_k(T)$ of $T$ at $T_k$} is a basic maximal rigid object obtained by replacing $T_k$ by another indecomposable object $T_k^*$.
The objects $T_k^*$ and $ T_k$ are related by the following {\it exchange triangles}
\[T_k^*\xrightarrow{f}U_{T_k, T\setminus T_k}\xrightarrow{g}T_k\to \Sigma T_k^* ~ \text{and}~ T_k\xrightarrow{f'} U'_{T_k, T\setminus T_k}\xrightarrow{g'}\ T_k^*\to\Sigma T_k,
\]
where $f$, $f'$ are minimal left $\add \overline{T}$-approximations and $g$, $g'$ are minimal right $\add \overline{T}$-approximations. In this case, $\overline{T}$ is called an {\it almost complete maximal rigid object} and $(T_k, T_k^*)$ is an {\it exchange pair} of $\cc$.

For each basic maximal rigid object $T=\bigoplus_{i=1}^nT_i$, we define a matrix $B_T=(b_{ij})\in M_{n}(\Z)$ as follows
\[b_{ij}=\mathfrak{m}_{U_{T_j, T\setminus T_j}}T_i-\mathfrak{m}_{U'_{T_j, T\setminus T_j}}T_i,
\]
where $\mathfrak{m}_YX$ denotes the multiplicity of $X$ as a direct summand of $Y$. 
The following observation interprets the matrix $B_T$ via the quiver $Q_T$ ({\it cf.} Lemma 2.12 of~\cite{FGL1}).
\begin{lemma}~\label{l:skew-symmetrizable-matrix-via-quiver}
Let $T=\bigoplus_{i=1}^nT_i$ be a basic maximal rigid object of $\cc$ with $l(T_1)=n$ and $Q_T$ its associated quiver. Let $B_T=(b_{ij})\in M_n(\Z)$ be the skew-symmetrizable matrix associated to $T$. Then for $i\neq j$, we have
\[b_{ij}=\begin{cases}|\{\text{arrows $T_i\to T_j$}\}|-|\{\text{arrows $T_j\to T_i$}\}|& j\neq 1;\\
2|\{\text{arrows $T_i\to T_1$}\}|-2|\{\text{arrows $T_1\to T_i$}\}|& j=1.
\end{cases}
\]
\end{lemma}

 The matrix $B_T$ is skew-symmetrizable and the associated cluster algebra $\mathcal{A}(B_T)$ is  of type $\mathrm{C}_n$~\cite{BuanMarshVatne}. It is known that the matrix $B_T$ and the matrix of the mutation $\mu_k(T)$ are related by Fomin-Zelevinsky's matrix mutation at $k$, {\it {i.e.}} $\mu_k(B_T)=B_{\mu_k(T)}$. Moreover,  the following result was proved~({\it cf.}~ Theorem 3.5 of ~\cite{BuanMarshVatne}).
\begin{theorem}~\label{t:BMV-main-theorem}
Let $T$ be a basic maximal rigid object of $\cc$ and $\mathcal{A}_T:=\mathcal{A}(B_T)$ the associated cluster algebra.
There is a bijection between the set of indecomposable rigid objects of $\cc$ and the set of cluster variables of $\mathcal{A}_T$. The bijection induces a bijection between the basic maximal rigid objects of $\cc$ and  the clusters of $\mathcal{A}_T$ such that $\Sigma T$ corresponds to the initial cluster of $\mathcal{A}_T$. Moreover, the bijection is compatible with mutations.
\end{theorem}

\subsection{Another construction of the matrix $B_T$}
In this subsection, we give another construction of the matrix $B_T$ for a basic maximal rigid object $T=\bigoplus_{i=1}^nT_i$. Recall that $A=\End_\cc(T)$ is the endomorphism algebra of $T$.
Let  $S_i$ be the simple $A$-module associated to the indecomposable direct summand  $T_i$.
Let $W$ be the sum of the cube of the unique loop $\rho$ and all $3$-cycles of $Q_T$ which gives a potential on $Q_T$. Then $A$ is isomorphic to the {\it Jacobian algebra} $ J(Q_T,W)$ associated to the quiver with potential $(Q_T,W)$ ({\it cf.}~Section 6~ in \cite{Yang12} for details).
To the quiver with potential $(Q_T,W)$, one may construct a so-called {\it Ginzburg dg algebra} $\Gamma:=\Gamma_{(Q_T,W)}$ and we refer to ~\cite{Amiot} for the precisely construction of Ginzburg dg algebras and the Jacobian algebras associated to quivers with potentials ({\it cf.} also~\cite{Ginzburg}). 

 Let $\der(\Gamma)$ be the derived category of $\Gamma$ with suspension functor $\Sigma $ and $\der_{fd}(\Gamma)$ the finite dimensional derived category of $\Gamma$. Namely, $\der_{fd}(\Gamma)$ is the full subcategory of $\der(\Gamma)$ of dg $\Gamma$-modules whose homology is of finite total dimension, which is  a 3-Calabi-Yau triangulated category ({\it cf.~\cite{Keller}}).
Denote by $\go(\der_{fd}(\Gamma))$ the Grothendieck group of $\der_{fd}(\Gamma)$. Let $\langle-, -\rangle_{\Gamma}:\go(\der_{fd}(\Gamma))\times \go(\der_{fd}(\Gamma))\to \C$ be the Euler bilinear form of $\go(\der_{fd}(\Gamma))$ given by
\[\langle [X], [Y]\rangle_\Gamma:=\sum_{i\in \Z}(-1)^i\dim_\C\Hom_{\der_{fd}(\Gamma)}(X, \Sigma^iY),
\]
where $X, Y\in \der_{fd}(\Gamma)$.
Note that $\Gamma$ is a non-positive dg algebra, which implies that there is canonical bounded $t$-structure on $\der_{fd}(\Gamma)$ whose heart is equivalent to the category $\mod A$~({\it cf.}~\cite{Amiot}). 
 In particular, the Euler form $\langle-,-\rangle_{\Gamma}$ induces an anti-symmetric Euler bilinear form $\langle-,-\rangle_a:\go(\mod A)\times \go(\mod A)\to \C$ of $\go(\mod A)$. More precisely, for any $M, N\in \mod A$, we may regard $M,N$ as objects of $\der_{fd}(\Gamma)$ and we have
\begin{eqnarray*}
&&\langle [M], [N]\rangle_a:=\langle[M],[N]\rangle_\Gamma=\sum_{i\in \Z}(-1)^i\dim_\C\Hom_{\der_{fd}(\Gamma)}(M,\Sigma^i N)\\
&=&\dim_\C\Hom_A(M,N)-\dim_\C\Ext_A^1(M,N)-\dim_\C\Hom_A(N,M)+\dim_\C\Ext_A^1(N,M),
\end{eqnarray*}
where the last equality follows from the $3$-Calabi-Yau property of $\der_{fd}(\Gamma)$.

We also introduce the {\it truncated Euler form} $\langle-,-\rangle_{\leq 1}$ on $\mod A$. For any $M,N\in \mod A$, we define
\[\langle M, N\rangle_{\leq 1}:=\dim_\C\Hom_A(M,N)-\dim_\C\Ext^1_A(M,N).
\]
It is clear that $\langle [M], [N]\rangle_a=\langle M,N\rangle_{\leq 1}-\langle N, M\rangle_{\leq 1}$.
By definition of $B_T$ and  Lemma~\ref{l:skew-symmetrizable-matrix-via-quiver}, we clearly have the following characterization of $B_T$ via the anti-symmetric Euler bilinear form $\langle-,-\rangle_a$.
\begin{proposition}~\label{p:skew-symmetrizable-matrix-via-skew-symmetric-Euler-form}
Let $T=\bigoplus_{i=1}^nT_i$ be a basic maximal rigid object of $\cc$ with $l(T_1)=n$ and $B_T=(b_{ij})\in M_n(\Z)$ the skew-symmetrizable matrix associated to $T$. Then 
\[b_{ij}=\begin{cases}~\langle S_i, S_j\rangle_a{,}& j\neq 1;\\
2\langle S_i,S_j\rangle_a{,}& j=1.\end{cases}
\]
\end{proposition}

\section{The main theorem}~\label{s:main-thm}

Let $T=\bigoplus_{i=1}^nT_i$ be a basic maximal rigid object of $\cc$ with $l(T_1)=n$. Let $A=\End_{\cc}(T)$ be the endomorphism algebra of $T$ and $\mod A$ the category of finitely generated right $A$-modules. Denote by $\tau_A$ the AR translation of $\mod A$.
Let $e_1,\cdots, e_n$ be the primitive idempotents of $A$ corresponding to $T_1,\cdots, T_n$ respectively. Denote by $I_i$ the indecomposable injective $A$-module associated to $e_i$.  Recall that we have an equivalence $F:\opname{pr}T/\add \Sigma T\to \mod A$.
\subsection{Locally free modules }

An $A$-module $M$ is {\it locally free} if $Me_i$ is free as a right $e_iAe_i$-module for each $i=1,\cdots, n$.  It is known that the subcategory of locally free $A$-modules is closed under extensions.
For a locally free $A$-module $M$, the integer vector
\[\rankv M=(r(Me_1),\cdots, r(Me_n))^{\opname{tr}}\in \Z^n
\]
is called the {\it rank vector} of $M$, where $r(Me_i)$ is the rank of $Me_i$ as a free $e_iAe_i$-module.

\begin{remark}~\label{r:locally-free}
According to Theorem~\ref{t:algebra-maximal-tube} and the assumption that $l(T_1)=n$. We have \[e_iAe_i\cong\End_\cc(T_i)\cong\begin{cases} \C[x]/(x^2)& i=1;\\ \C&\text{else.}\end{cases}\] Consequently, an $A$-module $M$ is locally free if and only if $Me_1$ is a free $e_1Ae_1$-module.
Moreover, if $M$ is a locally free $A$-module with dimension vector $\dimv M=(m_1, m_2, \cdots, m_n)^{\opname{tr}}$, then $\rankv M=(\frac{m_1}{2}, m_2, \cdots, m_n)^{\opname{tr}}$. By Lemma~\ref{l:dimension-vector-loop}, we have either $[\rankv M:1]=0$ or $[\rankv M:1]=1$ for an indecomposable locally free $A$-module $M$.
\end{remark}

\begin{lemma}~\label{l:rigid-locally-free}
 Each indecomposable $\tau_A$-rigid $A$-module is locally free.
\end{lemma}
\begin{proof}
By Proposition~\ref{p:tau-rigid}, each indecomposable $\tau_A$-rigid $A$-module is isomorphic to $F(X)$ for some indecomposable rigid object $X\in \cc\backslash\add\Sigma T$.
 According to Remark~\ref{r:locally-free}, it suffices to show that $F(X)e_1=\Hom_\cc(T_1,X)$ is a free $\End_\cc(T_1)$-module. Each indecomposable rigid object $X$ lies either in $\cw_{\tau T_1}$ or $T_1^{\sqsubset}\cap ^{\sqsupset}\tau^2T_1$.
If $X\in \cw_{\tau T_1}$, then $\Hom_\cc(T_1, X)=0$ and hence $F(X)$ is a locally free $A$-module. 

Now assume that $X\in T_1^{\sqsubset}\cap ^{\sqsupset}\tau^2T_1$.  Recall that we have \[\Hom_\cc(T_1, X)=\Hom_{\der^b(\ct)}(T_1, X)\oplus \Hom_{\der^b(\ct)}(T_1, \tau^{-1}\Sigma X).\] In this case,  we have \[\dim_\C\Hom_{\der^b(\ct)}(T_1, X)=1=\dim_\C\Hom_{\der^b(\ct)}(T_1, \tau^{-1}\Sigma X).\] 
To show that $\Hom_\cc(T_1,X)$ is a free $\End_\cc(T_1)$-module, it suffices to show that each non-zero morphism in $\Hom_{\der^b(\ct)}(T_1, \tau^{-1}\Sigma X)$ factors through a morphism in $\Hom_{\der^b(\ct)}(T_1,\tau^{-1}\Sigma T_1)$.
Let $f\in \Hom_\ct(T_1, X)$ be a non-zero morphism.
By Lemma~2.11(iii) of~\cite{FGL1}, $f$ induces an isomorphism \[\Hom_\ct(X, \tau^2 T_1)\cong \Hom_\ct(T_1, \tau^2 T_1)\] which fits into the following commutative diagram
\[\xymatrix{\Hom_\ct(X, \tau^2 T_1)\ar[d]^{\cong}\ar[rr]^{\Hom_\ct(f, \tau^2T_1)}&&\Hom_\ct(T_1, \tau^2 T_1)\ar[d]^{\cong}\\
D\Hom_{\der^b(\ct)}(\tau T_1, \Sigma X)\ar[rr]^{D\Hom_{\der^b(\ct)}(\tau T_1, \Sigma f)}&&D\Hom_{\der^b(\ct)}(\tau T_1, \Sigma T_1).
}
\]
Consequently, we have an isomorphism \[\Hom_{\der^b(\ct)}(T_1, \tau^{-1}\Sigma T_1)\xrightarrow{\Hom_{\der^b(\ct)}(T_1, \tau^{-1}\Sigma f)}\Hom_{\der^b(\ct)}(T_1, \tau^{-1}\Sigma X),\]
which implies the desired result.
\end{proof}

\begin{lemma}~\label{l:comparision-dimension}
Assume moreover that $T_1=(1,n)$.
For each $1\leq i<n$, the module $F((i,n+1))$ is locally free with
\[\dimv F((i, n+1))=\dimv F((i+1, n-1)).\] Consequently,  if 
$F((i,n+1))=M(a_1\cdots a_s \rho a_{s+1}\cdots a_r)$, then $F((i+1, n-1))=M(a_1\cdots a_s \rho^{-1} a_{s+1}\cdots a_r)$, where $a_1,\cdots, a_{r}$ are (formal) arrows of the quiver $Q_T$ and $\rho$ is the unique loop.
\end{lemma}
\begin{proof}
Since $\ct$ is standard, we have 
\[\dim_\C\Hom_{\der^b(\ct)}(T_1, (i,n+1))=1=\dim_\C\Hom_{\der^b(\ct)}(T_1,\tau^{-1}\Sigma (i,n+1)).
\]
Now similar to Lemma~\ref{l:rigid-locally-free}, one can show that $F((i,n+1))$ is locally free with $[\dimv F((i,n+1)):1]=2$.

According to Lemma~\ref{l:basic-property-cluster-tube} (2), 
each indecomposable direct summand $T_i$ belongs to $\cw_{(1,n)}$.
In order to prove $\dimv F((i, n+1))=\dimv F((i+1, n-1))$, it suffices to show that  for any indecomposable object $X\in \cw_{(1,n)}$, we have 
\[\dim_\C\Hom_{\cc}(X, (i,n+1))=\dim_\C\Hom_{\cc}(X, (i+1, n-1)).
\]
The set $\cw_{(1,n)}$ can be divided into the following  five disjoint subsets:
 $\cs_1:=\cw_{(1,n)}\backslash (\cw_{(1,i)}\cup \cw_{(i+1,n-i)})$, $\cs_2=:~^{\sqsupset}(i,1)\cap \cw_{(1,n)}$, $\cs_3:=(i+1,1)^{\sqsubset}\cap \cw_{(1,n)}$, $\cs_4:=\cw_{(1,i-1)}, \cs_5:=\cw_{(i+2, n-i-1)}$.
A direct computation shows that
\begin{itemize}
\item for any $X\in \cs_1$, $\dim_\C\Hom_{\cc}(X, (i, n+1))=2=\dim_\C\Hom_{\cc}(X, (i+1, n-1))$;
\item  for any $X\in \cs_2\cup \cs_3$, $\dim_\C\Hom_{\cc}(X, (i, n+1))=1=\dim_\C\Hom_{\cc}(X, (i+1, n-1))$;
\item for each $X\in \cs_4\cup \cs_5$, $\dim_\C\Hom_{\cc}(X, (i, n+1))=0=\dim_\C\Hom_{\cc}(X, (i+1, n-1))$.
\end{itemize}
The remain statement  follows from the description of $\mathcal{Q}_n$ in Section~\ref{ss: endomorphism-quiver}.
\end{proof}

\subsection{Grassmannians of locally free modules}

For an integer vector $\underline{e}=(a_1, a_2, \cdots, a_n)^{\opname{tr}}\in \N^n$, we set $\hat{\underline{e}}=(2a_1, a_2, \cdots, a_n)^{\opname{tr}}\in \N^n$.
For each locally free $A$-module $M$ and $\underline{e}\in \N^n$, we consider the following quasi-projective variety
\[\opname{Gr}_{\underline{e}}^{lf}(M):=\{N\subseteq M~|~\text{$N$ is a locally free submodule of $M$ with $\rankv N=\underline{e}$}\}.
\] 
It is clear that $\opname{Gr}_{\underline{e}}^{lf}(M)$~is a subvariety of 
\[\opname{Gr}_{\underline{\hat{e}}}(M):=\{N\subseteq M~|~\text{$N$ is a submodule of $M$ with $\dimv N=\underline{\hat{e}}$}\}.
\]
Moreover, if $M$ is indecomposable, then $\opname{Gr}_{\underline{e}}^{lf}(M)=\opname{Gr}_{\underline{\hat{e}}}(M)$ by Remark~\ref{r:locally-free}. We denote $\chi(\opname{Gr}_{\underline{e}}^{lf}(M))$ the Euler-Poincar\'{e} characteristic.  
The following  result has been established in~\cite{GLS17}.
\begin{lemma}~\label{l:multiplicativity-CC}
Let $M$ and $N$ be locally free $A$-modules. For $\underline{g}\in \N^n$, we have
\[\chi(\opname{Gr}_{\underline{g}}^{lf}(M\oplus N))=\sum_{\underline{e}+\underline{f}=\underline{g}}\chi(\opname{Gr}_{\underline{e}}^{lf}(M))\chi(\opname{Gr}_{\underline{f}}^{lf}(N)).
\]
\end{lemma}

As an application of Lemma~\ref{l:comparision-dimension}, we have
\begin{corollary}\label{c:equality-cc-map}
 Assume that $T_1=(1,n)$.
 For each $\underline{e}\in \N^n$ and $1\leq i<n$, we have \[\chi(\opname{Gr}_{\underline{e}}^{lf}(F((i,n+1))))=\chi(\opname{Gr}_{\underline{e}}^{lf}(F((i+1,n-1)))).\]
\end{corollary}
\begin{proof}
By Lemma~\ref{l:comparision-dimension} and the construction of indecomposable string modules({\it cf.} Section~\ref{ss: endomorphism-quiver}), $F((i,n+1))$ and $F((i+1,n-1))$ can be written as

\[\xymatrix{\\
F((i,n+1)):\cdots \ar[r]& V_1\ar[r]^{f_{21}}&\C\oplus \C\ar@{->}@(ur,ul)_{\tiny\begin{bmatrix}0&0\\1&0\end{bmatrix}}\ar[r]^{f_{13}}&V_3\ar[r]&\cdots\\
F((i+1,n-1)):\cdots \ar[r]& V_1\ar[r]^{f_{21}}&\C\oplus \C\ar@{->}@(dl,dr)_{\tiny\begin{bmatrix}0&1\\0&0\end{bmatrix}}\ar[r]^{f_{13}}&V_3\ar[r]&\cdots.
}
\]

Now it is not hard to see that for each $\underline{e}\in \N^n$
\[\opname{Gr}_{\underline{e}}^{lf}(F((i,n+1)))\cong \opname{Gr}_{\underline{e}}^{lf}(F((i+1,n-1)))\]and hence $\chi(\opname{Gr}_{\underline{e}}^{lf}(F((i,n+1))))=\chi(\opname{Gr}_{\underline{e}}^{lf}(F((i+1,n-1)))).$
\end{proof}

\subsection{The Caldero-Chapoton map}
Let $x_1,\cdots, x_n$ be $n$ indeterminants and $\mathbb{Q}[x_1^{\pm},\cdots, x_n^{\pm}]$ the  ring of Laurent polynomials in variables $x_1,\cdots, x_n$ with coefficients in $\mathbb{Q}$. For an integer vector $\alpha=(m_1, \cdots, m_n)^{\opname{tr}}\in \Z^n$ or  an element $\alpha=m_1[T_1]+\cdots+m_n[T_n]\in \go(\add T)$, we write
\[x^\alpha:=\prod_{i=1}^nx_i^{m_i}.
\]

Let $B_T$ be the skew-symmetrizable matrix associated to $T$ ({\it cf.} Section~\ref{s: skew-symmetrizable-matrix}). Let $x_1, x_2, \ldots, x_n$ be the initial cluster variables.
 Denote by  $\ca_T:=\ca(B_T)$ the cluster algebra without coefficients associated to $T$.  For a locally free  $A$-module $M$ with a minimal injective copresentation
 \[0\to M\to \bigoplus_{i=1}^nI_i^{\oplus a_i}\to \bigoplus_{i=1}^nI_i^{\oplus b_i},
\]
we define $\mathfrak{i}(M):=(a_1-b_1,\cdots, a_n-b_n)^{\opname{tr}}\in \Z^n$. The {\it Caldero-Chapoton map} or Palu's {\it cluster character} $\mathbb{X}_M$ of $M$ (with respect to $T$) is defined as follows
\[\mathbb{X}_M=x^{-\mathfrak{i}(M)}\sum_{\underline{e}}\chi(\opname{Gr}_{\underline{e}}^{lf}(M))x^{B_T\underline{e}}\in \mathbb{Q}[x_1^{\pm},\cdots, x_n^{\pm}].
\]
The following is the main result of this paper and its proof will be postponed to  Section~\ref{s:the-proof}.
\begin{theorem}~\label{t:main-thm}
Let $T$ be a basic maximal rigid object of $\cc$ with endomorphism algebra $A$. Denote by $\mathcal{A}_T$ the cluster algebra associated to $T$. The Caldero-Chapoton map $\mathbb{X}_?$ yields a bijection between the indecomposable $\tau_A$-rigid $A$-modules and the non-initial cluster variables of $\mathcal{A}_T$.
\end{theorem}
According to the {\it Laurent phenomenon}~\cite{FominZelevinsky02}, each cluster variable $\mathbb{X}_M$ is a Laurent polynomial in the initial cluster variables~$x_1,\cdots, x_n$. Thus there exists a unique polynomial $f(x_1,\cdots, x_n)$ which is not divisible by any $x_i$ such that 
\[\mathbb{X}_M=\frac{f(x_1,\cdots, x_n)}{x_1^{d_1}\cdots x_n^{d_n}}.
\]
The {\it denominator vector} of $\mathbb{X}_M$ is defined to be 
\[\opname{den}(\mathbb{X}_M)=(d_1,\cdots,d_n)^{\opname{tr}}\in \Z^n.
\]

As an immediately  consequence of Theorem~1.3 of~\cite{FGL1} and Theorem~\ref{t:main-thm}, we obtain
\begin{corollary}
For each indecomposable $\tau_A$-rigid $A$-module $M$, we have
\[\opname{den}(\mathbb{X}_M)=\underline{\opname{\rank}}~M.
\]
\end{corollary}

\section{Index and coindex}~\label{s:index-coindex}
\subsection{Reformulation of the main result}~\label{ss:index-coindex}
Recall that $\cc$ is a cluster tube of rank $n+1$.   For a set $\mathcal{X}$ consisting of certain objects of $\cc$, we define 
\[\Sigma \mathcal{X}:=\{\Sigma X~|~X\in \mathcal{X}\}.
\]

Let $T=\bigoplus_{i=1}^nT_i$ be a basic  maximal rigid object of $\cc$. 
Recall that $\opname{pr} T$ is the full subcategory of $\cc$ consisting of objects which are finitely presented by $T$.
 Set
\[\mathcal{F}:=\{(a, b)~|~b\leq n\} \cup \{(a, b)~|~n+1\leq b\leq 2n, a+b\leq 2n+1\}.
\]
Note that $\Sigma T$ is also a basic maximal rigid object of $\cc$.
As shown in~\cite{Yang12}, if $T$ lies in the wing $\cw_{(1,n)}$, then an indecomposable object $X$ lies in $\opname{pr}T$ if and only if $X\in \mathcal{F}$.  Moreover, an indecomposable object $X$ lies in $\opname{pr}\Sigma T$ if and only if $X\in \Sigma \mathcal{F}$~({\it cf.} also~\cite{Vatne11}). Define
\[\mathcal{R}:=\{(a, b)~|~b\leq n\}. 
\]
In particular, $\mathcal{R}$ consists of indecomposable rigid objects of $\cc$. It is clear that 
\[\Sigma \mathcal{R}=\mathcal{R}\subset \opname{pr}T\cap \opname{pr}\Sigma T.
\]
Let $\go(\add T)$ be the split Grothendieck group of $\add T$.
For each $X\in \opname{pr}T$, we have a triangle
\[T_1^X\to T_0^X\to X\to \Sigma T_1^X
\]
with $T_1^X, T_0^X\in \add T$.
The {\it index} $\opname{ind}_T(X)$ of $X$ with respect to $T$ is defined to be
\[\opname{ind}_T(X):=[T_0^X]-[T_1^X]\in \go(\add T),
\]
where $[*]$ stands for the image of the object $*$ in the Grothendieck group $\go(\add T)$. Similarly, for each object $X\in \opname{pr}\Sigma T$, there is a triangle
\[X\to \Sigma^2T_X^1\to \Sigma^2T_X^0\to \Sigma X
\]
with $T_X^0, T_X^1\in \add T$. The {\it coindex} $\opname{coind}_T(X)$ of $X$ with respect to $T$ is defined to be
\[\opname{coind}_T(X):=[T_X^1]-[T_X^0]\in \go(\add T).
\]
It is not hard to see that the index and coindex are well-defined. In particular, for each object $X\in \opname{pr}T\cap \opname{pr}\Sigma T$,  $\opname{ind}_T(X)$ and $\opname{coind}_T(X)$ are defined.  For an object $X\in \opname{pr}T\cap \opname{pr}\Sigma T$ such that $F(X)$ is locally free, we define the {\it Caldero-Chapoton map} or Palu's {\it cluster character} $\mathbb{X}_?^T$ of  $X$ with respect to $T$ as follows
\[\mathbb{X}_X^T=x^{-\opname{coind}_T(X)}\sum_{\underline{e}}\chi(\opname{Gr}_{\underline{e}}^{lf}(F(X)))x^{B_T\underline{e}}\in \mathbb{Q}[x_1^{\pm},\cdots,x_n^{\pm}].
\]
If $X$ is an indecomposable rigid object in $\cc\backslash\add \Sigma T$, we clearly have $\opname{coind}_T(X)=\mathfrak{i}(F(X))$. Consequently, $\mathbb{X}_{F(X)}=\mathbb{X}^T_X$ in this case. On the other hand,   we also have
\[\mathbb{X}_{\Sigma T_i}^T=x_i~\text{for $i=1,\cdots, n$,} 
\]
which are precisely the initial cluster variables of $\mathcal{A}_T$.
The main result Theorem~\ref{t:main-thm} can be restated as follows.
\begin{theorem}~\label{t:main-thm-restate}
Let $T$ be a basic maximal rigid object of $\cc$ and $\mathcal{A}_T$ the associated cluster algebra. The Caldero-Chapoton map $\mathbb{X}^T_?$ yields a bijection between the indecomposable rigid objects of $\cc$ and the cluster variables of $\mathcal{A}_T$.
\end{theorem}

\subsection{Behavior of index and coindex}
In this subsection, we establish some basic properties for index and coindex with respect to triangles in $\cc$. Throughout this section, we fix a basic maximal rigid object $T=\bigoplus_{i=1}^nT_i$ of $\cc$ with $T_1=(1,n)$.
We begin with the following result which is a direct consequence of the definitions and the fact that $\mathcal{R}=\Sigma \mathcal{R}$ ({\it cf.} Lemma 2.1 of~\cite{Palu} ).
\begin{lemma}
For any $X\in \mathcal{R}$, we have $\opname{ind}_T(X)=-\opname{coind}_T(\Sigma X)$.
\end{lemma}

Recall that $A=\End_{\cc}(T)$ is the endomorphism algebra of $T$ and $\mod A$ is the category of finitely generated right $A$-modules. Let $S_1, \cdots, S_n$ be the simple $A$-modules associated to $T_1,\cdots, T_n$ respectively. Denote by $P_i$ ({\it resp.} $I_i$) the projective cover ({\it resp.} injective hull) of $S_i$.
We have a truncated Euler form $\langle-,-\rangle_{\leq 1}$ on $\mod A$ and an equivalence 
\[F=\Hom_{\cc}(T,-):\opname{pr} T/\add \Sigma T\to \mod A.
\]
The following result has been proved in~\cite{Palu} ({\it cf.}~ Lemma~2.3 of~\cite{Palu}).
\begin{lemma}~\label{l:index-coindex-euler-form}
Let $X$ be an indecomposable object of $\cc$.
If $X\in \opname{pr}T$, then
\[\opname{ind}_T(X)=\begin{cases}-[T_i],& X\cong \Sigma T_i;\\ \sum_{i=1}^n\langle F(X), S_i\rangle_{\leq 1} [T_i],& \text{else}.\end{cases}
\]
If $X\in \opname{pr}\Sigma T$, then 
\[\opname{coind}_T(X)=\begin{cases}-[T_i],& X\cong \Sigma T_i;\\ \sum_{i=1}^n\langle S_i, F(X)\rangle_{\leq 1} [T_i],& \text{else}.\end{cases}
\]
\end{lemma}
\begin{corollary}~\label{c:coindex-index-matrix}
Let $X\in \opname{pr}T\cap\opname{pr}\Sigma T$ such that $F(X)$ is a locally free $A$-module, 
the coordinate vector of $\opname{coind}_T(X)-\opname{ind}_T(X)$ with respect to the basis $[T_1],\cdots, [T_n]$ of $\go(\add T)$ is $B_T\rankv F(X)$.
\end{corollary}
\begin{proof}
It suffices to prove the statement for indecomposable objects.
For an indecomposable object $X\in \opname{pr}T\cap \opname{pr}\Sigma T$, the index $\opname{ind}_T(X)$ and the coindex $\opname{coind}_T(X)$ are well-defined.  It is clear that the equality holds for $X\cong \Sigma T_i$ by Lemma~\ref{l:index-coindex-euler-form}. Now assume that $X\not\cong \Sigma T_i$ for any $i$. Again by Lemma~\ref{l:index-coindex-euler-form}, we obtain
\[\opname{coind}_T(X)-\opname{ind}_T(X)=\sum_{i=1}^n\langle F(X), S_i\rangle_a[T_i]=([T_1],\cdots, [T_n])B_T\rankv F(X),
\]
where the last equality follows from Proposition~\ref{p:skew-symmetrizable-matrix-via-skew-symmetric-Euler-form}.
\end{proof}

We now turn to the behavior of index and coindex with respect to triangles.
The following proposition  is a special case of Proposition 2.2 of ~\cite{Palu}. 
\begin{proposition}~\label{p:index-coindex-triangle}
Let $X\xrightarrow{f}Y\xrightarrow{g}Z\xrightarrow{h}\Sigma X$ be a triangle in $\cc$.
\begin{itemize}
\item If $F(g)$ is surjective and $X, Y, Z\in \opname{pr} T$, then $\opname{ind}_T(Y)=\opname{ind}_T(X)+\opname{ind}_T(Z)$; 
\item If $F(f)$ is injective and $X, Y, Z\in \opname{pr}\Sigma T$, then $\opname{coind}_T(Y)=\opname{coind}_T(X)+\opname{coind}_T(Z)$.
\end{itemize}
\end{proposition}

\begin{proposition}~\label{p:index-coindex-via-ARtriangle}
Let $\Sigma X\xrightarrow{f}Y\xrightarrow{g}X\xrightarrow{h} \Sigma ^2X$ be an AR-triangle in $\cc$ with $X\in \mathcal{R}$ and $Y\in \opname{pr}T\cap\opname{pr}\Sigma T$.
\begin{itemize}
\item[(a)] if neither $X$ nor $\Sigma X$ belongs to $\add \Sigma T$, then 
 $0\to F(\Sigma X)\xrightarrow{F(f)}F(Y)\xrightarrow{F(g)}F(X)\to 0$ 
 is an AR-sequence of $\mod A$. Consequently, \[\opname{ind}_T(Y)=\opname{ind}_T(X)+\opname{ind}_T(\Sigma X)~ \text{and}~ \opname{coind}_T(Y)=\opname{coind}_T(X)+\opname{coind}_T(\Sigma X).\]
 \item[(b)] if $X\cong \Sigma T_k$ for some $k$, then $k\neq 1$ and $F(Y)$ is the maximal locally free factor of $F(\Sigma X)=I_k$. Moreover, the coordinate vector of $\opname{coind}_T(Y)$ with respect to  $[T_1],\cdots, [T_n]$ is $-B_T\underline{e_k}$.
 \item[(c)] if $\Sigma X\cong \Sigma T_k$ for some $k$, then $k\neq 1$ and $F(Y)$ is the maximal locally free submodule of $F(X)=P_k$. Consequently,  $\opname{coind}_T(T_k)=\opname{coind}_T(Y)+\opname{coind}_T(\Sigma^2T_k)$.
\end{itemize}
\end{proposition}
\begin{proof}
For $(a)$, by the equivalence $F:\opname{pr} T/\add \Sigma T\to \mod A$ and Proposition~\ref{p:index-coindex-triangle}, it suffices to show that 
\[0\to F(\Sigma X)\xrightarrow{F(f)}F(Y)\xrightarrow{F(g)}F(X)\to 0\]
is a short exact sequence. Equivalently, $F(g)$ is surjective and $F(f)$ is injective.
 Note that $\Sigma X\not\in \add\Sigma T$ implies that $X\not\in \add T$. Therefore each morphism from $T$ to $X$ is not a retraction. We then deduce that $F(g):F(Y)\to F(X) $ is surjective by the fact that $g:Y\to X$ is a right almost split morphism in $\cc$.  Let $t: T\to \Sigma X$ be a morphism such that $f\circ t=0$. Then  the morphism $t$ factors through $\Sigma^{-1}h$. Namely, there is a morphism $s: T\to \Sigma^{-1}X$ fitting into the following commutative diagram
\[\xymatrix{& T\ar[d]^t\ar[dl]_s&& \Sigma T\ar[d]^{\Sigma s}\ar[dr]^{\Sigma t}\\
\Sigma^{-1}X\ar[r]^{\Sigma^{-1}h}&\Sigma X\ar[r]^f&Y\ar[r]^g&X\ar[r]^{h}&\Sigma^2X.
}
\]
Recall that we also have $X\not\in \add \Sigma T$, which implies that $\Sigma s$ is not a retraction. Therefore $\Sigma t=h\circ \Sigma s=0$. Consequently, $t=0$ and $F(f)$ is injective.

Let us consider the statement $(b)$.
  According to the description of $\opname{pr}T$ and $\opname{pr}\Sigma T$ in Section~\ref{ss:index-coindex}, it is not hard to see that $k\neq 1$ by $Y\in \opname{pr}T\cap  \opname{ pr}\Sigma T$. Applying the functor $F$, we obtain an exact sequence of $A$-modules
\[P_k\xrightarrow{F(\Sigma^{-1}h)} I_k\xrightarrow{F(f)} F(Y)\to 0.
\]
It suffices to show that $\im F(\Sigma^{-1}h)=S_k$. Note that 
\[\Sigma T_k\xrightarrow{\Sigma^{-1}f} \Sigma^{-1}Y\xrightarrow{\Sigma^{-1}g}T_k\xrightarrow{\Sigma^{-1}h}\Sigma^2 T_k\] 
is also an AR-triangle in $\cc$. In particular, for $j\neq k$ and any morphism $t:T_j\to T_k$,  the composition $\Sigma^{-1}h\circ t$ vanishes. On the other hand, we clearly have $\dim_\C\Hom_{\cc}(T_k,T_k)=1$. Therefore $\im F(\Sigma^{-1}h)=S_k$ and $F(Y)$ is the maximal locally free factor of $I_k$. Moreover, $\dimv F(Y)=\dimv I_k-\dimv S_k$.
Using Lemma~\ref{l:index-coindex-euler-form}, we have the following computation
\begin{eqnarray*}
\opname{coind}_T(Y)-\opname{ind}_T(Y)&=&\sum_{i=1}^n\langle S_i,F(Y)\rangle_a[T_i]\\
&=&-\sum_{i=1}^n\langle S_i, S_k\rangle_a[T_i]+\sum_{i=1}^n\langle S_i, I_k\rangle_a[T_i]\\
&=&-([T_1],\cdots, [T_n])B_T\underline{e_k}+\opname{coind}_T(\Sigma^2 T_k)-\opname{ind}_T(\Sigma^2 T_k).\end{eqnarray*}
On the other hand,  by Proposition~\ref{p:index-coindex-triangle},  we have
$\opname{ind}_T(\Sigma^2 T_k)=\opname{ind}_T(T_k)+\opname{ind}_T(Y)$. It follows that $\opname{coind}_T(Y)=-([T_1],\cdots, [T_n])B_T\underline{e_k}$.

For $(c)$, it is also easy to deduce that $k\neq 1$ by $Y\in \opname{pr}T\cap \opname{pr}\Sigma T$. Applying the functor $F$ to the triangle $\Sigma T_k\xrightarrow{f} Y\xrightarrow{g} T_k\xrightarrow{h}\Sigma^2 T_k$, we obtain an exact sequence
\[0\to F(Y)\xrightarrow{F(g)} P_k\xrightarrow{F(h)} I_k.
\]
Similar to the case $(b)$, one can  show that the image of $F(h)$ is the simple module $S_k$. Therefore $F(Y)$ is the maximal locally free submodule of $P_k$. The remaining statement is a direct consequence of Proposition~\ref{p:index-coindex-triangle}.
\end{proof}

For $k\neq 1$, Proposition~\ref{p:index-coindex-via-ARtriangle} gives a characterization of the maximal locally free submodule ({\it resp.} factor) of the indecomposable projective ({\it  resp.} injective) $A$-module $P_k$ ({\it resp. $I_k$}) via the structure of $\cc$. Recall that we have $T_1=(1,n)$. For the case $k=1$, we have the following.

\begin{lemma}~\label{l:exchange-triangle-1}
\begin{itemize}
\item[(1)]~$F((1, n-1)\oplus (1, n-1))$ is the maximal locally free submodule of $F((1,n))=P_1$;
\item[(2)]~$F((n+1,n-1)\oplus(n+1,n-1))$ is the maximal locally free factor of $F((n,n))=I_1$.
\end{itemize}
\end{lemma}
\begin{proof}
We prove  the statement $(1)$ and the proof of  $(2)$ is similar. Let 
\[\Sigma T_1\to (1, n-1)^{\oplus 2}\xrightarrow{g} T_1\xrightarrow{h} \Sigma^2T_1\]
be one of the exchange triangles associated to the basic maximal rigid objects $M=(1,1)\oplus\cdots (1,n-1)\oplus (1,n)$ and $N=(1,1)\oplus\cdots\oplus (1,n-1)\oplus (n+1,n)$ ({\it cf.}~\cite{ZhouZhu14}).
Applying the functor $F$ to the triangle above, we obtain an exact sequence of $A$-modules
\[0\to F((1,n-1)^{\oplus 2})\xrightarrow{F(g)} P_1\xrightarrow{F(h)} I_1.
\]
It suffices to show that $\dimv\im F(h)=2\underline{e_1}$. For any $i\neq 1$, noticing that $T_i\in \cw_{T_1}$, it is not hard to show that any morphism $T_i\to T_1$ factors through the morphism $g$ ({\it cf.}~ Lemma 2.11 of ~\cite{FGL1}). On the other hand,  any non-zero morphism $t:T_1\to T_1$ does not factor through $g$ and we conclude that $\dimv \im F(h)=2\underline{e_1}$.
\end{proof}

\begin{lemma}\label{l:equality-cc-map}
 Let $1\leq i<n$.
 We have \[\opname{ind}_T(i,n+1)=\opname{ind}_T(i+1, n-1)~\text{and}~ \opname{coind}_T(i,n+1)=\opname{coind}_T(i+1, n-1).\] Consequently, $\mathbb{X}^T_{(i,n+1)}=\mathbb{X}^T_{(i+1,n-1)}$.
\end{lemma}
\begin{proof}
  For each $1\leq i<n$, we have the following two triangles
 \[(i,n)\to (i+1,n-1)\oplus (i,n+1)\to (i+1,n)\to \Sigma (i,n) \]
 \[(i,n)\to (i+1,n-1)^{\oplus 2}\to (i+1,n)\to \Sigma (i,n).
 \]
 Note that the first triangle is an AR-triangle in $\cc$.  By  Proposition~\ref{p:index-coindex-via-ARtriangle} (a), we know that 
 \[0\to F((i,n))\to F((i+1,n-1))\oplus F((i,n+1))\to F((i+1,n))\to 0
 \]
 is a short exact sequence in $\mod A$. Moreover, \[\opname{ind}_T(i+1.n-1)+\opname{ind}_T (i,n+1)=\opname{ind}_T(i,n)+\opname{ind}_T(i+1,n)\] and \[\opname{coind}_T(i+1.n-1)+\opname{coind}_T(i,n+1)=\opname{coind}_T(i,n)+\opname{coind}_T(i+1,n).\] 
 
 Applying the functor $F$ to the second triangle yields a long exact sequence
 \[\cdots\to F((i,n))\to F((i+1,n-1)^{\oplus 2})\to F((i+1,n))\to\cdots. 
 \]
 According to Lemma~\ref{l:comparision-dimension}, we have $\dimv F((i+1,n-1))=\dimv F((i,n+1))$. In particular, 
 \[0\to F((i,n))\to F((i+1,n-1)^{\oplus 2})\to F((i+1,n))\to 0
 \]
 is also a short exact sequence. Consequently, 
 \[2\opname{ind}_T(i+1,n-1)=\opname{ind}_T(i,n)+\opname{ind}_T(i+1,n)\] and \[2\opname{coind}_T(i+1,n-1)=\opname{coind}_T(i,n)+\opname{coind}_T(i+1,n),\]which imply the desired equalities. By Corollary~\ref{c:equality-cc-map}, we have $\mathbb{X}^T_{(i,n+1)}=\mathbb{X}^T_{(i+1,n-1)}$.
 \end{proof}

\section{Grassmannians of locally free submodules}~\label{s:grassmanian}
Throughout this section, we fix a basic maximal rigid object $T=\bigoplus_{i=1}^nT_i$ of $\cc$ with $T_1=(1,n)$. Denote by $A=\End_{\cc}(T)$ the endomorphism algebra of $T$. Recall that we have an equivalence $F:\opname{pr} T/\add \Sigma T\to \mod A$. The aim of this section is to establish certain basic properties for Grassmannians of locally free submodules.

\subsection{Identification of $A$-modules with representations}Let $Q=(Q_0,Q_1)$ be the opposite Gabriel quiver of $A$ with $Q_0=\{1,\cdots,n\}$ the set of vertices and $Q_1$ the set of arrows. 
 By the assumption $l(T_1)=n$, we know that vertex $1$ is the unique vertex associated to the unique loop $\rho$. Let $I$ be the ideal of the path algebra $kQ$ generated by the square of the unique loop and all paths of length $2$ in a $3$-cycle. Then $\mod A$ is equivalent to the category $\opname{rep}(Q, I)$ of finite dimensional representations of the quiver with relation $(Q, I)$. In the following, we will identify an $A$-module with a representation of $(Q, I)$. 

Let $X$ be an indecomposable object in $\opname{pr}T\backslash \add \Sigma T$ such that $F(X)$ is locally free. By abuse of notations, we still denote by $F(X)=(F(X)_i, F(X)_\alpha)_{i\in Q_0, \alpha\in Q_1}$ the representation in $\opname{rep}(Q, I)$ which corresponds to the module $F(X)\in \mod A$. It is not hard to see  that $F(X)_i=\Hom_{\cc}(T_i,X)$. For any arrow $\alpha:i\to j$ in $Q$, there is a morphism $\alpha: T_j\to T_i$ in $\cc$, then $F(X)_\alpha=\Hom_{\cc}(\alpha, X)$.
Moreover, if $\Hom_{\cc}(T_1,X)\neq 0$, then $\dim_\C\Hom_{\cc}(T_1, X)=2$. In this case, $F(X)_1$ admits a canonical decomposition of one-dimensional vector spaces as follows:
\begin{eqnarray*}
F(X)_1&=&\Hom_{\der^b(\ct)}(T_1, X)\oplus \Hom_{\der^b(\ct)}(\tau\circ \Sigma^{-1}T_1, X),\\
F(X)_{\rho}&=&\Hom_{\der^b(\ct)}(T_1, X)\oplus \Hom_{\der^b(\ct)}(\tau\circ \Sigma^{-1}T_1, X)\\
&&\xrightarrow{\tiny\begin{bmatrix}0&0\\ \Hom_{\cc}(\rho, X)&0\end{bmatrix}}\Hom_{\der^b(\ct)}(T_1, X)\oplus \Hom_{\der^b(\ct)}(\tau\circ \Sigma^{-1}T_1, X).
\end{eqnarray*}
We remark that $\Hom_{\cc}(\rho, X)$ is a linear isomorphism of one-dimensional $\C$-vector spaces. 
Let $Y$ be an indecomposable object of $\opname{pr}T\backslash \add\Sigma T$ such that $F(Y)$ is locally free with $\dim_\C\Hom_\cc(T_1,Y)=2$ and $f:X\to Y$ a $\ct$-morphism. Then $f$ induces a homomorphism of representations $F(f):F(X)\to F(Y)$ such that the linear map $F(f)_1$ has the following presentation
\begin{eqnarray*}F(f)_1&=&\Hom_{\der^b(\ct)}(T_1,X)\oplus \Hom_{\der^b(\ct)}(\tau\Sigma^{-1}T_1,X)\\&&\xrightarrow{\tiny\begin{bmatrix}\Hom_{\der^b(\ct)}(T_1,f)&0\\ 0&\Hom_{\der^b(\ct)}(\tau\Sigma^{-1}T_1, f)\end{bmatrix}}\Hom_{\der^b(\ct)}(T_1,Y)\oplus \Hom_{\der^b(\ct)}(\tau\Sigma^{-1}T_1,Y).
\end{eqnarray*}
Similarly, if $g:X\to Y$ is a $\der$-morphism, we have
\begin{eqnarray*}F(g)_1&=&\Hom_{\der^b(\ct)}(T_1,X)\oplus \Hom_{\der^b(\ct)}(\tau\Sigma^{-1}T_1,X)\\&&\xrightarrow{\tiny\begin{bmatrix}0&0\\ \Hom_{\der^b(\ct)}(T_1,g)&0\end{bmatrix}}\Hom_{\der^b(\ct)}(T_1,Y)\oplus \Hom_{\der^b(\ct)}(\tau\Sigma^{-1}T_1,Y).
\end{eqnarray*}
\subsection{A $\mathbb{C}^{\times}$-action on Grassmannians of locally free submodules}
Let $M$ be a locally free $A$-module and $
\underline{g}$ an integer vector in $\N^n$. Recall that there is a stratification of morphisms in $\cc$ given by $\ct$-morphisms and $\mathcal{D}$-morphisms.
 The aim of this subsection is to show that the stratification of morphisms in $\cc$ yields a $\mathbb{C}^{\times}$-action on the Grassmannian $\opname{Gr}_{\underline{g}}^{lf}(M)$ of locally free submodules.  
 
 Let us first recall an interpretation of points of $\opname{Gr}_{\underline{g}}^{lf}(M)$ as morphisms. Fix a $U\in \opname{Gr}_{\underline{g}}^{lf}(M)$. For each $U'\in \opname{Gr}_{\underline{g}}^{lf}(M)$ such that $U'\cong U$, the point $U'$ can be represented by an injective morphism $\varphi:U\to M$ of $A$-modules such that $\im \varphi=U'$. Moreover, if $\varphi':U\to M$ is another injective morphism with $\im \varphi'=U'$, then there exists an automorphism $\phi: U\to U$ such that $\varphi'=\varphi\circ \phi$. 
 
  Denote by $\hat{U}$ and $\hat{M}$ the preimages of $U$ and $M$ in $\opname{pr}T\backslash\add\Sigma T$ respectively. Recall that we have an equivalence $F:\opname{pr}T/\add\Sigma T\to \mod A\cong \rep(Q,I)$. In the following, we identify $M$ with $F(\hat{M})$. Let $h_U:\hat{U}\to \hat{M}$ be a morphism in $\cc$ such that $F(h_U)$ is injective with $\im F(h_U)=U\subset M=F(\hat{M})$. Since each morphism in $\cc$ can be written as a sum of a $\ct$-morphism with a $\der$-morphism, we may write $h_U=h_{U,0}+h_{U,1}$, where $h_{U,0}:\hat{U}\to \hat{M}$ is a $\ct$-morphism and $h_{U,1}:\hat{U}\to \hat{M}$ is a $\der$-morphism. We then define
\begin{eqnarray*}
\bullet: &\C^\times \times \opname{Gr}_{\underline{g}}^{lf}(M)&\to \opname{Gr}_{\underline{g}}^{lf}(M)\\
&(c, U)&\mapsto \im F(h_{U,0}+ch_{U,1}).
\end{eqnarray*}

\begin{lemma}~\label{l:c-action}
The map $\bullet$ defines a $\C^\times$-action on $\opname{Gr}_{\underline{g}}^{lf}(M)$.
\end{lemma}
\begin{proof}
It suffices to show that $\im F(h_{U,0}+ch_{U,1})\in \opname{Gr}_{\underline{g}}^{lf}(M)$ and the definition is independent of the choice of the morphism $h_U$.

 In order to prove $\im F(h_{U,0}+ch_{U,1})\in \opname{Gr}_{\underline{g}}^{lf}(M)$, it suffice to show that $F(h_{U,0}+ch_{U,1})$ is injective. Let $g:=g_0+g_1:T_i\to \hat{U}$ be a morphism such that $(h_{U,0}+ch_{U,1}) g=0$, where $g_0$ is a $\ct$-morphism and $g_1$ is a $\der$-morphism. Note that the composition of two $\der$-morphisms is zero. We have
\[0=(h_{U,0}+ch_{U,1})g=h_{U,0}g_0+ch_{U,1}g_0+h_{U,0}g_1.
\]
In particular, $h_{U,0}g_0=0$ and $ch_{U,1}g_0+h_{U,0}g_1=0$.  Consequently, $h_{U,0}(cg_0)=0$ and $h_{U,1}(cg_0)+h_{U,0}g_1=0$, which implies that $(h_{U,0}+h_{U,1})(cg_0+g_1)=0$. Note that, as $F(h_{U,0}+h_{U,1})$ is injective, we have $cg_0+g_1=0$ and hence $g_0=0=g_1$.
That is, $F(h_{U,0}+ch_{U,1})$ is injective.

Let $f_U=f_{U,0}+f_{U,1}:\hat{U}\to \hat{M}$ be another morphism such that $F(f_U)$ is injective with $\im F(f_U)=U$, where $f_{U,0}$ is a $\ct$-morphism and $f_{U,1}$ is a $\der$-morphism. We need to show that $\im F(f_{U,0}+cf_{U,1})=\im F(h_{U,0}+ch_{U,1})$. Note that $\im F(f_U)=\im F(h_U)$ implies that there is an automorphism $\phi:F(\hat{U})\to F(\hat{U})$ such that $F(f_U)=F(h_U)\circ \phi$. Let $t=t_0+t_1:\hat{U}\to \hat{U}$ be a lift of $\phi$, where $t_0$ is a $\ct$-morphism and $t_1$ is a $\der$-morphism. Namely, $F(t)=\phi$.  
Then we obtain
\[F(f_U)=F(h_U)F(t)=F(h_Ut),
\]
which implies that there exists a morphism $s:\hat{U}\to\hat{M}$ factorizing through $\add \Sigma T$ such that $f_U=h_Ut+s$. We can also rewrite $s$ as $s=s_0+s_1$ such that $s_0$ is a $\ct$-morphism and $s_1$ is a $\der$-morphism. It is clear that both $s_0$ and $s_1$ factor through $\add \Sigma T$. Putting all of these together, we compute
\[f_{U,0}+f_{U,1}=(h_{U,0}+h_{U,1})(t_0+t_1)+s_0+s_1=(h_{U,0}t_0+s_0)+(h_{U,0}t_1+h_{U,1}t_0+s_1).
\]
In particular, $f_{U,0}=h_{U,0}t_0+s_0$ and $f_{U,1}=h_{U,0}t_1+h_{U,1}t_0+s_1$. Consequently, \[\im F(f_{U,0}+cf_{U,1})=\im F(h_{U,0}t_0+s_0+ch_{U,0}t_1+ch_{U,1}t_0+cs_1)=\im F(h_{U,0}t_0+ch_{U,0}t_1+ch_{U,1}t_0).\]
Recall that $\phi=F(t_0+t_1)$ is an isomorphism of $F(\hat{U})$. Similar to the above discussion, one can show that $F(t_0+ct_1)$ is also an isomorphism of $F(\hat{U})$. Hence we have 
\[\im F(f_{U,0}+cf_{U,1})=\im F((h_{U,0}+ch_{U,1})(t_0+ct_1))=\im F(h_{U,0}+ch_{U,1})F(t_0+ct_1)=\im F(h_{U,0}+ch_{U,1}).
\]
\end{proof}

\subsection{An auxiliary Grassmannian associated to AR-sequences}
Let $0\to L\xrightarrow{\iota}M\xrightarrow{\kappa}N\to 0$ be an AR sequence of $\mod A$ such that $N$ is $\tau_A$-rigid.  By Proposition~\ref{p:tau-rigid}, we may assume that $N=F((a,b))$ for some $1\leq b\leq n$ such that neither $(a,b)$ nor $(a-1,b)$ belongs to $\add \Sigma T$. Consequently, the AR sequence is the image of the AR triangle
\[(a-1,b)\xrightarrow{\tiny\begin{bmatrix}\hat{\iota}_1\\ \hat{\iota}_2\end{bmatrix}} (a-1,b+1)\oplus (a,b-1)\xrightarrow{\tiny\begin{bmatrix}\hat{\kappa}_1&\hat{\kappa}_2\end{bmatrix}} (a,b)\to \Sigma (a-1,b)
\]
of $\cc$ by Proposition~\ref{p:index-coindex-via-ARtriangle} (a). In particular, both $L$ and $M$ are locally free.
Note that each AR triangle of $\cc$ is the image of an AR triangle of $\der^b(\ct)$ under the projection $\pi:\der^b(\ct)\to \cc$. Therefore, without loss of generality, we may assume that $\hat{\iota}_1, \hat{\iota}_2$, $\hat{\kappa}_1,\hat{\kappa}_2$ are $\ct$-morphisms and hence $\iota$ and $\kappa$ are images of $\ct$-morphisms. For simplicity, we will also denote by $\hat{N}=(a,b)$, $\hat{X}=(a-1,b+1)$, $\hat{Y}=(a,b-1)$ and $\hat{M}=\hat{X}\oplus \hat{Y}$.

For each $\underline{g}\in \N^n$, we introduce the following auxiliary Grassmannian of $M$ with respect to the morphism $\iota$ and $\kappa$:
\[\opname{Gr}_{\underline{g}}^{lf}(M,\iota,\kappa):=\{ U\in \opname{Gr}_{\underline{g}}^{lf}(M)~|~\text{both $\iota^{-1}(U)$ and $\kappa(U)$ are locally free}\}.
\]
\begin{proposition}~\label{p:fibre-grassmanian}
Keep the notations as above. For each $\underline{g}\in \N^n$, consider the following morphism
\begin{eqnarray*}
\zeta_{\underline{g},{lf}}: &\opname{Gr}_{\underline{g}}^{lf}(M, \iota,\kappa)&\to \coprod_{\underline{e}+\underline{f}=\underline{g}}\opname{Gr}_{\underline{e}}^{lf}(L)\times \opname{Gr}_{\underline{f}}^{lf}(N)\\
&U&\mapsto (\iota^{-1}(U), \kappa(U)).
\end{eqnarray*}
Then the fibre $\zeta_{\underline{g},{lf}}^{-1}(B,C)$ is  an affine space isomorphic to $\Hom_A(C, L/B)$ if $(B,C)\neq (0, N)$ and $\zeta_{\underline{g},{lf}}^{-1}(0,N)=\emptyset$.
\end{proposition}
\begin{proof}
  For each $\underline{g}\in \N^n$, consider the following morphism of varieties
\begin{eqnarray*}
\zeta_{\underline{\hat{g}}}: &\opname{Gr}_{\underline{\hat{g}}}(M)&\to \coprod_{\underline{h}+\underline{k}=\underline{\hat{g}}}\opname{Gr}_{\underline{h}}(L)\times \opname{Gr}_{\underline{k}}(N)\\
&U&\mapsto (\iota^{-1}(U), \kappa(U)).
\end{eqnarray*}
It has been proved in~\cite{DG} that the fibre $\zeta_{\underline{\hat{g}}}^{-1}(B,C)$ is an affine space isomorphic to $\Hom_A(C, L/B)$ if $(B,C)\neq (0,N)$ and $\zeta_{\underline{\hat{g}}}^{-1}(0,N)=\emptyset$.
Note that we have $\opname{Gr}_{\underline{g}}^{lf}(M,\iota,\kappa)\subseteq \opname{Gr}_{\underline{g}}^{lf}(M)\subseteq\opname{Gr}_{\underline{\hat{g}}}(M)$. Consider the restriction $\zeta_{\underline{\hat{g}}}|_{\opname{Gr}_{\underline{g}}^{lf}(M,\iota,\kappa)}$ of $\zeta_{\underline{\hat{g}}}$ to ${\opname{Gr}}_{\underline{g}}^{lf}(M,\iota,\kappa)$. It is clear that the image of $\zeta_{\underline{\hat{g}}}|_{\opname{Gr}_{\underline{g}}^{lf}(M,\iota,\kappa)}$ lies in $\coprod_{\underline{e}+\underline{f}=\underline{g}}\opname{Gr}_{\underline{e}}^{lf}(L)\times \opname{Gr}_{\underline{f}}^{lf}(N)$. Moreover, $\zeta_{\underline{g}, {lf}}$ coincides with $\zeta_{\underline{\hat{g}}}|_{\opname{Gr}_{\underline{g}}^{lf}(M,\iota,\kappa)}$. Clearly, we have $\zeta_{\underline{g},{lf}}^{-1}(0,N)=\emptyset$.

Let $(B,C)\in \opname{Gr}_{\underline{e}}^{lf}(L)\times \opname{Gr}_{\underline{f}}^{lf}(N)$ such that $(B,C)\neq (0,N)$ and $U$ an object in $\zeta_{\underline{\hat{g}}}^{-1}(B,C)$. We have a short exact sequence
\[0\to B\to U\to C\to 0.
\]
In particular, $U$ is a locally free $A$-module with rank vector $\rankv U=\underline{g}$. Consequently, $U\in \opname{Gr}_{\underline{g}}^{lf}(M,\iota,\kappa)$. This shows that $\zeta_{\underline{\hat{g}}}^{-1}(B,C)=\zeta_{\underline{g},{lf}}^{-1}(B,C)$ and hence is isomorphic to the affine space $\Hom_A(C,L/B)$.
 
\end{proof}

\begin{lemma}~\label{l:free-action}
Assume that $[\underline{g}:1]=1$ and $\hat{N}=(a,b)\in (2,n)^{\sqsubset}\cap^{\sqsupset}(n,n)$. If $\opname{Gr}_{\underline{g}}^{lf}(M)\backslash \opname{Gr}_{\underline{g}}^{lf}(M,\iota,\kappa)\neq \emptyset$, then the map $\bullet$ induces a free $\C^\times$-action on $\opname{Gr}_{\underline{g}}^{lf}(M)\backslash \opname{Gr}_{\underline{g}}^{lf}(M,\iota,\kappa)$.
\end{lemma}
\begin{proof}
By $\hat{N}\in (2,n)^{\sqsubset}\cap^{\sqsupset}(n,n)$, we deduce that $\hat{M}$ has exactly two indecomposable direct summands $\hat{X}$ and $\hat{Y}$. Moreover, we have $[\rankv M:1]=2$ and $[\rankv L:1]=[\rankv N:1]=1$ in this case.

For $U\in \opname{Gr}_{\underline{g}}^{lf}(M)\backslash \opname{Gr}_{\underline{g}}^{lf}(M,\iota,\kappa)$, denote by $\hat{U}$ the preimage of $U$ in $\opname{pr}T\backslash\add \Sigma T$. Let  $h_U={\tiny\begin{bmatrix}h_{\hat{X},0}+h_{\hat{X},1}\\ h_{\hat{Y},0}+h_{\hat{Y},1}\end{bmatrix}}:\hat{U}\to \hat{X}\oplus \hat{Y}=\hat{M}$ be a morphism such that $F(h_U)$ is injective with $\im F(h_U)=U\subseteq F(\hat{M})=M$, where $h_{\hat{X},0}, h_{\hat{Y},0}$ are $\ct$-morphisms and $h_{\hat{X},1}, h_{\hat{Y},1}$ are $\der$-morphisms.
As $[\underline{g}:1]=1$, we can decompose $\hat{U}$ as $\hat{U}=\hat{V}\oplus \hat{W}$ such that $\hat{V}$ is indecomposable and 
\[[\rankv F(\hat{V}):1]=1, ~[\rankv F(\hat{W}):1]=0.
\]
Consequently, $\hat{V}\in (1,n)^{\sqsubset}\cap ^{\sqsupset}(n,n)$ and $F(\hat{U})_1=F(\hat{V})_1=\Hom_{\der^b(\ct)}(T_1,\hat{V})\oplus \Hom_{\der^b(\ct)}(\tau\Sigma^{-1}T_1, \hat{V})$. The morphism $h_U$ induces a $\C$-linear map $F(h_U)_1$ as
\begin{eqnarray*}
F(h_U)_1&=&\Hom_{\der^b(\ct)}(T_1,\hat{V})\oplus \Hom_{\der^b(\ct)}(\tau\Sigma^{-1}T_1,\hat{V})\xrightarrow{\tiny\begin{bmatrix}\Hom_{\der^b(\ct)}(T_1, h_{\hat{X},0})&0\\ \Hom_{\der^b(\ct)}(T_1, h_{\hat{X},1})&\Hom_{\der^b(\ct)}(\tau\Sigma^{-1}T_1, h_{\hat{X},0})\\ \Hom_{\der^b(\ct)}(T_1, h_{\hat{Y},0})&0\\ \Hom_{\der^b(\ct)}(T_1, h_{\hat{Y},1})&\Hom_{\der^b(\ct)}(\tau\Sigma^{-1}T_1, h_{\hat{Y},0})\end{bmatrix}}\\
&&\Hom_{\der^b(\ct)}(T_1,\hat{X})\oplus \Hom_{\der^b(\ct)}(\tau\Sigma^{-1}T_1,\hat{X})\oplus\Hom_{\der^b(\ct)}(T_1,\hat{Y})\oplus \Hom_{\der^b(\ct)}(\tau\Sigma^{-1}T_1,\hat{Y}).
\end{eqnarray*}
As $F(h_{U})$ is injective, we conclude that at least one of $h_{\hat{X},0}$ and $h_{\hat{Y},0}$ is non-zero. Without loss of generality, let us assume that $h_{\hat{X},0}\neq 0$. Note that, as $\hat{V}$ and $\hat{X}$ belong to  $(1,n)^\sqsubset\cap ^\sqsupset(n,n)$, we have $\Hom_{\der^b(\ct)}(T_1, h_{X,0})\neq 0\neq \Hom_{\der^b(\ct)}(\tau\Sigma^{-1}T_1, h_{\hat{X},0})$ ({\it cf.} Lemma~2.10 in~\cite{FGL1}). In particular, both $\Hom_{\der^b(\ct)}(T_1, h_{X,0})$ and $\Hom_{\der^b(\ct)}(\tau\Sigma^{-1}T_1, h_{\hat{X},0})$ are isomorphisms of one-dimensional $\C$-vector spaces.

On the other hand, by the assumption that $\kappa:M\to N$ is the image of the $\ct$-morphism $[\hat{\kappa}_1~\hat{\kappa}_2]$, we also have a $\C$-linear map $F([\hat{\kappa}_1~\hat{\kappa}_2])_1$ as
\begin{eqnarray*}F([\hat{\kappa}_1~\hat{\kappa}_2])_1&=&\Hom_{\der^b(\ct)}(T_1,\hat{X})\oplus \Hom_{\der^b(\ct)}(\tau\Sigma^{-1}T_1,\hat{X})\oplus\Hom_{\der^b(\ct)}(T_1,\hat{Y})\oplus \Hom_{\der^b(\ct)}(\tau\Sigma^{-1}T_1,\hat{Y})\\
&&\xrightarrow{\tiny\begin{bmatrix}\Hom_{\der^b(\ct)}(T_1, \hat{\kappa}_1)&0&\Hom_{\der^b(\ct)}(T_1, \hat{\kappa}_2)&0\\ 0&\Hom_{\der^b(\ct)}(\tau\Sigma^{-1}T_1, \hat{\kappa}_1)&0&\Hom_{\der^b(\ct)}(\tau\Sigma^{-1}T_1, \hat{\kappa}_2)\end{bmatrix}} \\
&&\Hom_{\der^b(\ct)}(T_1, \hat{N})\oplus \Hom_{\der^b(\ct)}(\tau\Sigma^{-1}T_1, \hat{N}).
\end{eqnarray*}
Moreover, all of $\Hom_{\der^b(\ct)}(T_1, \hat{\kappa}_1)$, $\Hom_{\der^b(\ct)}(\tau\Sigma^{-1}T_1, \hat{\kappa}_1)$, $\Hom_{\der^b(\ct)}(T_1, \hat{\kappa}_2)$ and $\Hom_{\der^b(\ct)}(\tau\Sigma^{-1}T_1, \hat{\kappa}_2)$ are isomorphisms of one-dimensional $\C$-vector spaces. We may identify each one-dimensional $\Hom$-space with $\C$ by choosing a suitable basis such that the morphisms have the following presentations:
\[\xymatrix{F(\hat{U})_1=\C\oplus \C\ar@{->}@(ur,ul)_{\tiny\begin{bmatrix}0&0\\1&0\end{bmatrix}}\ar[dd]^{{\tiny\begin{bmatrix}1&0\\ x&1\\ a&0\\ y&a\end{bmatrix}}=F(h_{U})_1}\\
\\
F(\hat{M})_1=\C\oplus \C\oplus \C\oplus \C\ar@{->}@(dr,dl)^{\tiny\begin{bmatrix}0&0&0&0\\1&0&0&0\\ 0&0&0&0\\ 0&0&1&0\end{bmatrix}}\ar[rrrr]_{~\quad~\tiny\begin{bmatrix}1&0&-1&0\\ 0&1&0&-1\end{bmatrix}=F([\hat{\kappa}_1~\hat{\kappa}_2])_1}&&&&F(\hat{N})_1=\C\oplus \C.\ar@{->}@(ur,ul)_{\tiny\begin{bmatrix}0&0\\1&0\end{bmatrix}}.
}
\]
Consequently, we have
\[F([\hat{\kappa}_1~\hat{\kappa}_2])_1F(h_U)_1=\begin{bmatrix}1&0&-1&0\\ 0&1&0&-1\end{bmatrix}\begin{bmatrix}1&0\\ x&1\\ a&0\\ y&a\end{bmatrix}=\begin{bmatrix}1-a&0\\ x-y&1-a\end{bmatrix}.
\]
Note that $U\in \opname{Gr}_{\underline{g}}^{lf}(M)\backslash\opname{Gr}_{\underline{g}}^{lf}(M,\iota,\kappa)$ implies that $\kappa(U)$ is not a locally free $A$-module. Hence we have $1-a=0$ and $x-y\neq 0$. For $c\in \C^\times$, denote by $U':=c\bullet U$, it is clear that $U'_1=\im \begin{bmatrix}1&0\\ cx&1\\ 1&0\\ cy&1\end{bmatrix}: \C\oplus \C\to \C\oplus \C\oplus \C\oplus \C$. Consequently, 
\[\begin{bmatrix}1&0&-1&0\\ 0&1&0&-1\end{bmatrix}\begin{bmatrix}1&0\\ cx&1\\ 1&0\\ cy&1\end{bmatrix}=\begin{bmatrix}0&0\\ c(x-y)&0\end{bmatrix},
\]
which implies that $U'\in \opname{Gr}_{\underline{g}}^{lf}(M)\backslash\opname{Gr}_{\underline{g}}^{lf}(M,\iota,\kappa)$. In particular, the $\C^\times$-action $\bullet$ of $\opname{Gr}_{\underline{g}}^{lf}(M)$ induces a $\C^\times$-action on $\opname{Gr}_{\underline{g}}^{lf}(M)\backslash\opname{Gr}_{\underline{g}}^{lf}(M,\iota,\kappa)$.
Now,  it is routine to check that 
\[\im \begin{bmatrix}1&0\\ cx&1\\ 1&0\\ cy&1\end{bmatrix}\neq \im \begin{bmatrix}1&0\\ x&1\\ 1&0\\ y&1\end{bmatrix}
\]
for $c\neq 1$, 
which implies that the action is free.

\end{proof}

Now we are in position to state the main result of this subsection.
\begin{proposition}~\label{p:equality-Euler-char}
Keep the notations as above. For each $\underline{g}\in \N^n$, we have
\[\chi(\opname{Gr}_{\underline{g}}^{lf}(M))=\chi(\opname{Gr}_{\underline{g}}^{lf}(M,\iota,\kappa)).
\]
\end{proposition}

\begin{proof}
We first note that an indecomposable rigid object $Z\in \cc$ satisfies that $[\rankv F(Z):1]\neq 0$ if and only if $Z\in (1,n)^{\sqsubset}\cap ~^{\sqsupset}(n, n)$. Moreover, we have $[\rankv F(Z):1]=1$ in this case ({\it cf.} Lemma~\ref{l:morphism-non-vanish}).
Recall that $\hat{N}=(a,b)$ is rigid.

If $\hat{N}\in \cw_{(1,n)}$, then $\tau\hat{N}\in \cw_{(n+1,n)}$ and hence $[\rankv L:1]=0$. Moreover, we also have $[\rankv M:1]=[\rankv N:1]\leq 1$. Consequently, for each  $U\in \opname{Gr}_{\underline{g}}^{lf}(M)$, we deduce that both $\iota^{-1}(U)$ and $\kappa(U)$ are locally free.
In particular, $\opname{Gr}_{\underline{g}}^{lf}(M)=\opname{Gr}_{\underline{g}}^{lf}(M,\iota, \kappa)$ and hence $\chi(\opname{Gr}_{\underline{g}}^{lf}(M))=\chi(\opname{Gr}_{\underline{g}}^{lf}(M,\iota,\kappa))$  provided that $\hat{N}\in \cw_{(1,n)}$.

If $\hat{N}\in (n+1,1)^{\sqsubset}$, then $[\rankv N:1]=0$ and $[\rankv L:1]=[\rankv M:1]=1$. Similar to the above case, one can show that $\opname{Gr}_{\underline{g}}^{lf}(M)=\opname{Gr}_{\underline{g}}^{lf}(M,\iota, \kappa)$. Hence $\chi(\opname{Gr}_{\underline{g}}^{lf}(M))=\chi(\opname{Gr}_{\underline{g}}^{lf}(M,\iota,\kappa))$ in this case.

Now it remains to consider the case $\hat{N}\in (2,n)^{\sqsubset}\cap~^{\sqsupset}(n,n)$. In this case, we have $[\rankv L:1]=[\rankv N:1]=1$ and $[\rankv M:1]=2$. If $[\underline{g}:1]=0$ or $[\underline{g}:1]=2$, we clearly have $\opname{Gr}_{\underline{g}}^{lf}(M)=\opname{Gr}_{\underline{g}}^{lf}(M,\iota, \kappa)$. Now assume that $[\underline{g}:1]=1$. Applying Lemma~\ref{l:free-action}, we conclude that $\chi(\opname{Gr}_{\underline{g}}^{lf}(M))=\chi(\opname{Gr}_{\underline{g}}^{lf}(M,\iota,\kappa))$. This completes the proof.

\end{proof}

\section{The proof of the main theorem}~\label{s:the-proof}

\subsection{Exchange relations}~\label{ss: the proof of the main result}
We fix a basic maximal rigid object $T=\bigoplus_{i=1}^nT_i$ with $T_1=(1,n)$ throughout this subsection.
We begin with the multiplicativity of $\mathbb{X}^T_?$, which is a direct consequence of Lemma~\ref{l:multiplicativity-CC} and the definition of $\mathbb{X}^T_?$.
\begin{lemma}~\label{l:cc-map-direct-sum}
Let $X,Y\in \opname{pr}T\cap\opname{pr}\Sigma T$ such that $F(X)$ and $F(Y)$ are locally free, then
\[\mathbb{X}^T_{X\oplus Y}=\mathbb{X}^T_X\mathbb{X}^T_Y.
\]
\end{lemma}
 Let $R=(1,1)\oplus \cdots (1,n-1)\oplus (1,n)$ and $S=(1,1)\oplus\cdots (1,n-1)\oplus (n+1,n)$ be two basic maximal rigid object of $\cc$. There exist two exchange triangles associated to $R$ and $S$. Namely,
\[(1,n)\to 0\to (n+1,n)\to (n+1,n)
\]
and
\[(n+1,n)\to (1,n-1)^{\oplus 2}\to (1,n)\to (n,n).
\]
The following result computes the exchange relation between the Caldero-Chapoton maps of $T_1=(1,n)$ and $\Sigma T_1$.
\begin{proposition}~\label{p:projective-injective-case}
We have
\[\mathbb{X}_{\Sigma T_1}^T\mathbb{X}^T_{(1,n)}=1+(\mathbb{X}^T_{(1,n-1)})^2\]~\text{and}~\[\mathbb{X}_{\Sigma T_1}^T\mathbb{X}^T_{(n,n)}=1+(\mathbb{X}^T_{(n+1,n-1)})^2.
\]
\end{proposition}
\begin{proof}
By definition, we have $\mathbb{X}_{\Sigma T_1}^T=x_1$ and 
\[\mathbb{X}_{(1,n)}^T=x^{-\opname{coind}_T(1,n)}\sum_{\underline{e}}\chi(\opname{Gr}_{\underline{e}}^{lf}(P_1))x^{B_T\underline{e}}.
\]
Denote by $M=F((1,n-1))$.
According to Lemma~\ref{l:exchange-triangle-1}, we may rewrite
\[\mathbb{X}_{(1,n)}^T=x^{-\opname{coind}_T(T_1)}\sum_{\underline{f}}\chi(\opname{Gr}_{\underline{f}}^{lf}(M^{\oplus 2}))x^{B_T\underline{f}}+x^{-\opname{coind}_T(T_1)}x^{B_T\underline{\opname{rank}}P_1}.
\] 
On the other hand, by Proposition~\ref{p:index-coindex-triangle} we have \[\opname{coind}_TT_1-\opname{coind}_T(\Sigma^2T_1)=\opname{coind}_T(1,n-1)^{\oplus 2},\] and by Corollary~\ref{c:coindex-index-matrix} we have \[\opname{coind}_TT_1-\opname{ind}_TT_1=([T_1],\cdots, [T_n])B_T\rankv P_1.\] 
Putting all of these together, we compute
\begin{eqnarray*}
x_1\mathbb{X}_{(1,n)}^T&=&x^{\opname{coind}_T(\Sigma^2T_1)-\opname{coind}_TT_1}\sum_f\chi(\opname{Gr}_f^{lf}(M^{\oplus 2}))x^{B_Tf}+1\\
&=&1+x^{-\opname{coind}_T(1,n-1)^{\oplus 2}}\sum_{\underline{f}}\chi(\opname{Gr}_{\underline{f}}^{lf}(M^{\oplus 2}))x^{B_T\underline{f}}\\
&=&1+(\mathbb{X}_{(1,n-1)}^T)^2,
\end{eqnarray*}
where the last equality follows from Lemma~\ref{l:cc-map-direct-sum}.
This completes the proof of the first equality.
The second equality can be proved similarly.
\end{proof}

\begin{proposition}~\label{p:n-case}
For $1\leq c<n$, we have
\[\mathbb{X}^T_{(c,n)}\mathbb{X}^T_{(c+1,n)}=1+(\mathbb{X}_{(c+1,n-1)}^T)^2.
\]
\end{proposition}
\begin{proof}
Let \[(c,n)\to (c,n+1)\oplus (c+1,n-1)\to (c+1,n)\to (c-1,n)\] be the AR-triangle in $\cc$ associated to $(c,n)$. The condition $1\leq c<n$ implies that $(c,n+1)\in \opname{pr} T\cap \opname{pr}\Sigma T$. By Proposition~\ref{p:index-coindex-via-ARtriangle} (a),  we know that 
\[0\to F((c,n))\to F((c+1,n-1))\oplus F((c,n+1))\to F((c+1,n))\to 0\] is an AR sequence in $\mod A$ and \[\opname{coind}_T(c,n)+\opname{coind}_T(c+1,n)=\opname{coind}_T(c,n+1)+\opname{coind}_T(c+1,n-1).\]
Denote by $L=F((c,n)), M=F((c+1,n-1)), M'=F((c,n+1)), N=F((c+1,n))$,
we can compute
\begin{eqnarray*}
\mathbb{X}_{(c,n)}^T\mathbb{X}_{(c+1,n)}^T&=&x^{-\opname{coind}_T(c,n)-\opname{coind}_T(c+1,n)}\sum_{\underline{e}, \underline{f}}\chi(\opname{Gr}_{\underline{e}}^{lf}(L))\chi(\opname{Gr}_{\underline{f}}^{lf}(N))x^{B_T(\underline{e}+\underline{f})}\\
&\overset{(1)}=&x^{-\opname{coind}_T(c,n)-\opname{coind}_T(c+1,n)}(x^{B_T\rankv N}+\sum_{\underline{g}}\chi(\opname{Gr}_{\underline{g}}^{lf}(M\oplus M'))x^{B_T\underline{g}})\\
&\overset{(2)}=&1+x^{-\opname{coind}_T(c,n)-\opname{coind}_T(c+1,n)}\sum_{\underline{g}}\chi(\opname{Gr}_{\underline{g}}^{lf}(M\oplus M'))x^{B_T\underline{g}}\\
&=&1+x^{-\opname{coind}_T(c,n+1)-\opname{coind}_T(c+1,n-1)}\sum_{\underline{g}}\chi(\opname{Gr}_{\underline{g}}^{lf}(M\oplus M'))x^{B_T\underline{g}}\\
&\overset{(3)}=&1+\mathbb{X}^T_{(c+1,n-1)}\mathbb{X}^T_{(c,n+1)}\\
&\overset{(4)}=&1+(\mathbb{X}^T_{(c+1,n-1)})^2,
\end{eqnarray*}
where the equality $(1)$ follows from Proposition~\ref{p:fibre-grassmanian} and ~\ref{p:equality-Euler-char} and the equality $(2)$ is a consequence of Corollary~\ref{c:coindex-index-matrix}, for the equalities $(3)$ and $(4)$, we use Lemma~\ref{l:cc-map-direct-sum} and Lemma~\ref{l:equality-cc-map} respectively.
\end{proof}

\begin{proposition}~\label{p:non-n-case}
For $1\leq b<n$, we have
\[\mathbb{X}_{(a,b)}^T\mathbb{X}_{(a+1,b)}^T=1+\mathbb{X}_{(a+1,b-1)}^T\mathbb{X}_{(a,b+1)}^T.
\]
\end{proposition}
\begin{proof}
Let \[(a,b)\to (a,b+1)\oplus (a+1,b-1)\to (a+1,b)\to (a-1,b)\] be the AR triangle in $\cc$. 
We separate the proof into three cases.
\begin{itemize}
\item[{\bf Case 1:}] Neither $(a,b)$ nor $(a+1,b)$ belongs to $\add \Sigma T$. 
\end{itemize}

In this situation, we deduce that \[0\to F((a,b))\to F((a,b+1))\oplus F((a+1,b-1))\to F((a+1,b))\to 0\] is an AR sequence of $A$-modules and
\[\opname{coind}_T(a,b)+\opname{coind}_T(a+1,b)=\opname{coind}_T(a,b+1)+\opname{coind}_T(a+1,b-1)
\]
by Proposition~\ref{p:index-coindex-via-ARtriangle}~(a). Similar to Proposition~\ref{p:n-case}, one can prove the required equality.
\begin{itemize}
\item[{\bf Case 2:}] $(a+1,b)\in \add \Sigma T$.
\end{itemize}
 
 We may assume $(a,b)=\Sigma^2T_k$ for $k\neq 1$ and $(a+1,b)=\Sigma T_k$. According to Proposition~\ref{p:index-coindex-via-ARtriangle}~(b),  $F(a,b+1)\oplus F(a+1,b-1)$ is the maximal locally free factor of the injective module $F(a,b)=I_k$. Moreover, $\opname{coind}_T(a,b+1)\oplus (a+1,b-1)=-B_T\underline{e_k}$. Using this, we compute
\begin{eqnarray*}
\mathbb{X}^T_{(a,b)}\mathbb{X}^T_{(a+1,b)}&=&x^{\underline{e_k}-\opname{coind}_T(a,b)}\sum_{\underline{e}}\chi(\opname{Gr}_{\underline{e}}^{lf}(I_k))x^{B_T\underline{e}}\\
&=&1+\sum_{\underline{f}}\chi(\opname{Gr}_{\underline{f}}^{lf}(F((a,b+1))\oplus F((a+1,b-1)))x^{B_T(\underline{f}+\underline{e_k})}\\
&=&1+x^{-\opname{coind}_T(a,b+1)\oplus (a+1,b-1)}\sum_{\underline{f}}\chi(\opname{Gr}_{\underline{f}}^{lf}(F((a,b+1))\oplus F((a+1,b-1)))x^{B_T\underline{f}}\\
&=&1+\mathbb{X}_{(a+1,b-1)}^T\mathbb{X}_{(a,b+1)}^T.
\end{eqnarray*}
\begin{itemize}
\item[{\bf Case 3:}] $(a,b)\in \add \Sigma T$.
\end{itemize}

 Without loss of generality, we may assume $(a,b)=\Sigma T_k$ for $k\neq 1$ and hence $\mathbb{X}_{(a,b)}^T=x_k$.
 By Proposition~\ref{p:index-coindex-via-ARtriangle}~(c), we know that $F((a,b+1))\oplus F((a+1,b-1))$ is the maximal locally free submodule of the projective module$F((a+1,b))=P_k$ and 
 \[\opname{coind}_T(a+1,b)=\underline{e_k}+\opname{coind}_T(a,b+1)\oplus(a+1,b-1).
 \] 
 We compute
 \begin{eqnarray*}
 \mathbb{X}^T_{(a,b)}\mathbb{X}^T_{(a+1,b)}&=&x^{\underline{e_k}-\opname{coind}_T(a+1,b)}\sum_{\underline{e}}\chi(\opname{Gr}_{\underline{e}}^{lf}(P_k))x^{B_T\underline{e}}\\
 &=&x^{\underline{e_k}-\opname{coind}_T(a+1,b)}(x^{B_T\rankv P_k}+\sum_{\underline{f}}\chi(\opname{Gr}_{\underline{f}}^{lf}(F((a,b+1))\oplus F((a+1,b-1)))x^{B_T\underline{f}})\\
 &\overset{(1)}=&1+x^{-\opname{\coind}_T(a,b+1)\oplus (a+1,b-1)}\sum_{\underline{f}}\chi(\opname{Gr}_{\underline{f}}^{lf}(F((a,b+1))\oplus F((a+1,b-1)))x^{B_T\underline{f}}\\
&\overset{(2)}=&1+\mathbb{X}^T_{(a,b+1)}\mathbb{X}^T_{(a+1,b-1)},
 \end{eqnarray*}
where $(1)$ follows from Corollary~\ref{c:coindex-index-matrix} and $(2)$ is a consequence of Lemma~\ref{l:cc-map-direct-sum}.
\end{proof}

\subsection{Proof of the Main Theorem:}
It suffices to prove Theorem~\ref{t:main-thm-restate}.
Let $T=T_1\oplus\cdots\oplus T_n$ be a  basic maximal rigid object of $\cc$ with $T_1=(1,n)$ and $\mathcal{A}_{T}$ the associated cluster algebra without coefficients.
Let $\mathbb{T}_n$ be the $n$-regular tree. We fix a cluster pattern for the cluster algebra $\mathcal{A}_{T}$ by assigning the initial cluster $\mathrm{x}=\{x_1,\cdots, x_n\}$ to the root vertex $t_{\opname{r}}$. On the other hand, we may also assign each vertex $t\in \mathbb{T}_n$ a basic maximal rigid object $T_t=\bigoplus_{i=1}^nT_{t, i}$ such that
\begin{itemize}
\item[$\circ$] $T_{t_{\opname{r}}}=\Sigma T_1\oplus \cdots\oplus \Sigma T_n$.
\item[$\circ$] If $t\frac{~~k~~}{}t'$ is an edge of $\mathbb{T}_n$ labeled by $k$, then $\mu_k(T_{t})=T_{t'}$.
\end{itemize}
By Theorem~\ref{t:BMV-main-theorem}, there is a bijection between the basic maximal rigid objects of $\cc$ and the clusters of $\ca_{T}$, which is compatible with the mutations. 
Note that $\mathcal{A}_{T}$ has only finitely many cluster variables and by the definition of the Caldero-Chapoton map, we already have $\mathbb{X}^T_{\Sigma T_i}=x_i$. Thus, for each indecomposable rigid object $M\in \cc$, to show $\mathbb{X}^T_M$ is the cluster variable associated to $M$, it suffices to find a finite path in $\mathbb{T}_n$ such that
\begin{enumerate}
\item[(a)] each indecomposable rigid object appears as a direct summand for some basic maximal rigid objects associated to the vertices on the path;
\item[(b)] $\mathbb{X}^T_M$ is compatible with the exchange relations appears in the path.
\end{enumerate}
Denote by 
\[\xymatrix{t_0\ar@{-}[r]^1&t_1\ar@{-}[r]^2&\cdots\ar@{-}[r]^n&t_n\ar@{-}[r]^{1}&t_{n+1}\ar@{-}[r]^{2}&t_{n+2}\cdots\ar@{-}[r]^n&t_{n^2}}
\]
a path of the $n$-regular tree $\mathbb{T}_n$, where
\[T_{t_0}=(n+1,1)\oplus (n+1,2)\oplus \cdots\oplus (n+1,n)\]and for each $1\leq i\leq n^2$,  \[T_{t_i}=(a+1,1)\oplus \cdots\oplus (a+1,b)\oplus(a,b+1)\oplus \cdots\oplus (a,n),\] where $i=na+b, 0\leq b<n$.
It is clear that the path satisfies the condition $(a)$.
By applying Proposition~\ref{p:projective-injective-case},~\ref{p:n-case} and~\ref{p:non-n-case}, we conclude that $\mathbb{X}_M^T$ is compatible with the exchange traingles appears in the above path and we are done. \hfill$\Box$

\subsection{An example}
In this subsection, we give a non-acyclic example to illustrate our result.
Let $A$ be the algebra as in Example~\ref{ex:string-module}. Let $\mod A$ be the category of finitely generated right $A$-modules. Let $\tau_A$ be the AR translation of $\mod A$. 
The AR quiver of $\mod A$  is shown in Figure~1, where the indecomposable $\tau_A$-rigid modules are marked with rectangles. In particular, there are exactly $9$ indecomposable $\tau_A$-rigid $A$-modules with different rank vectors.
\begin{figure}
\[\xymatrix@=5pt{
&&&*[F]{{\tiny\begin{array}{c}2\\1\\1\\3\end{array} }}\ar[dr]&&&&*[F]{{\tiny\begin{array}{c}3\\2\end{array} }}\ar[dr]&&&&*[F]{{\tiny\begin{array}{c}2\\1\\1\\3\end{array} }}\ar[dr]&&&\\
&&*[F]{{\tiny\begin{array}{c}1\\1\\3\end{array} }}\ar[dr]\ar[ur]&&*[F]{{\tiny\begin{array}{c}2\\1\\1\end{array}} }\ar[dr]&&*[F]{{\tiny\begin{array}{c}2\end{array} } }\ar[ur]&&*[F]{{\tiny\begin{array}{c}3\end{array} }}\ar[dr]&&*[F]{{\tiny\begin{array}{c}1\\1\\3\end{array}} }\ar[dr]\ar[ur]&&*[F]{{\tiny\begin{array}{c}2\\1\\1\end{array}} }\ar[dr]&&\\
&*[F]{{\tiny\begin{array}{c}1\\13\\3\end{array} }}\ar[dr]\ar[ur]&&*[F]{{\tiny\begin{array}{c}1\\1\end{array}} }\ar[dr]\ar[ur]&&*[F]{{\tiny\begin{array}{c}2\\12\\1\end{array}} }\ar[dr]\ar[ur]&&&&*[F]{{\tiny\begin{array}{c}1\\13\\3\end{array} }}\ar[dr]\ar[ur]&&*[F]{{\tiny\begin{array}{c}1\\1\end{array} } }\ar[dr]\ar[ur]&&*[F]{{\tiny\begin{array}{c}2\\12\\1\end{array}} }\ar[ur]&\\
{\tiny\begin{array}{c}1\\3\end{array} }\ar[dr]\ar[ur]&&{\tiny\begin{array}{c}1\\13\end{array} }\ar[dr]\ar[ur]&&{\tiny\begin{array}{c}21\\1\end{array} }\ar[dr]\ar[ur]&&{\tiny\begin{array}{c}2\\1\end{array} }\ar[dr]&&{\tiny\begin{array}{c}1\\3\end{array} }\ar[dr]\ar[ur]&&{\tiny\begin{array}{c}1\\13\end{array} }\ar[dr]\ar[ur]&&{\tiny\begin{array}{c}21\\1\end{array} }\ar[dr]\ar[ur]&&\\
&{\tiny\begin{array}{c}1\end{array} }\ar[dr]\ar[ur]&&{\tiny\begin{array}{c}21\\13\end{array} }\ar[dr]\ar[ur]&&{\tiny\begin{array}{c}1\end{array} }\ar[dr]\ar[ur]&&{\tiny\begin{array}{c}21\\13\end{array} }\ar[dr]\ar[ur]&&{\tiny\begin{array}{c}1\end{array} }\ar[dr]\ar[ur]&&{\tiny\begin{array}{c}21\\13\end{array} }\ar[dr]\ar[ur]&&{\tiny\begin{array}{c}1\end{array} }\ar[dr]\ar[ur]&\\
&&{\tiny\begin{array}{c}2\\1\end{array} }\ar[ur]&&{\tiny\begin{array}{c}1\\3\end{array} }\ar[dr]\ar[ur]&&{\tiny\begin{array}{c}1\\13\end{array} }\ar[dr]\ar[ur]&&{\tiny\begin{array}{c}21\\1\end{array} }\ar[dr]\ar[ur]&&{\tiny\begin{array}{c}2\\1\end{array} }\ar[ur]&&{\tiny\begin{array}{c}1\\3\end{array} }\ar[dr]\ar[ur]&&{\tiny\begin{array}{c}1\\13\end{array} }\\
&*[F]{{\tiny\begin{array}{c}2\\12\\1\end{array} }}\ar[dr]\ar[ur]&&&&*[F]{{\tiny\begin{array}{c}1\\13\\3\end{array} }}\ar[dr]\ar[ur]&&*[F]{{\tiny\begin{array}{c}1\\1\end{array} } }\ar[dr]\ar[ur]&&*[F]{{\tiny\begin{array}{c}2\\12\\1\end{array} }}\ar[dr]\ar[ur]&&&&*[F]{{\tiny\begin{array}{c}1\\13\\3\end{array} }}\ar[dr]\ar[ur]&\\
&&*[F]{{\tiny\begin{array}{c}2\end{array} }}\ar[dr]&&*[F]{{\tiny\begin{array}{c}3\end{array} }}\ar[ur]&&*[F]{{\tiny\begin{array}{c}1\\1\\3\end{array} } }\ar[dr]\ar[ur]&&*[F]{{\tiny\begin{array}{c}2\\1\\1\end{array} }}\ar[ur]&&*[F]{{\tiny\begin{array}{c}2\end{array} }}\ar[dr]&&*[F]{{\tiny\begin{array}{c}3\end{array}} }\ar[ur]&&\\
&&&*[F]{{\tiny\begin{array}{c}3\\2\end{array} }}\ar[ur]&&&&*[F]{{\tiny\begin{array}{c}2\\1\\1\\3\end{array} }}\ar[ur]&&&&*[F]{{\tiny\begin{array}{c}3\\2\end{array}} }\ar[ur]&&&
}\]
\caption{The AR quiver of $\mod A$.}
\end{figure}
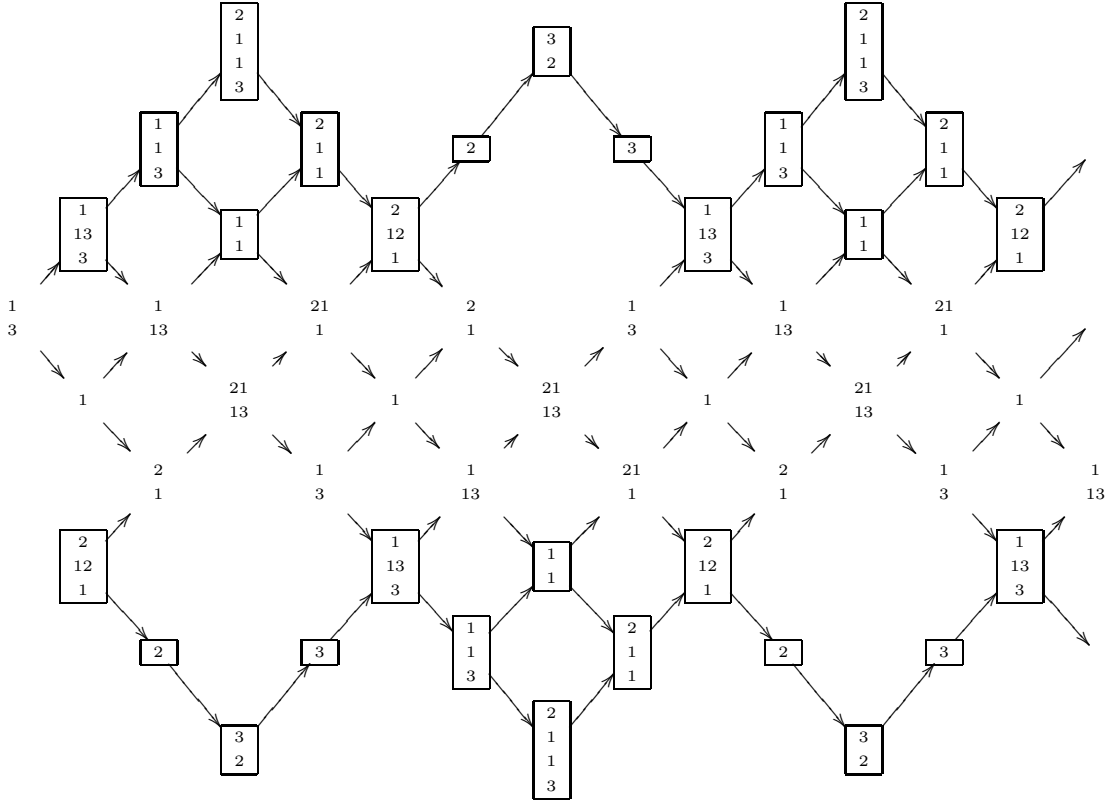

Set \[\tiny B=\begin{bmatrix}0 &1&-1\\ -2&0&1\\ 2&-1&0\end{bmatrix}\] and let $\mathcal{A}=\mathcal{A}(B)$ be the cluster algebra associated to $B$, which is a cluster algebra of type $\mathrm{C}_3$.

Denote by $M_{\underline{c}}$ the unique indecomposable $\tau_A$-rigid $A$-module with the rank vector $\underline{c}$~\footnote{In this example, we use row vector for simplicity.}. Let $I_1, I_2, I_3$ be the indecomposable injective $A$-modules associated to the vertices $1,2,3$, respectively. 
For example, it is easy to obtain the minimal injective resolution of $M_{(1,0,2)}$:
\[0\to M_{(1,0,2)}\to I_3^{\oplus 2}\to I_1.
\]
Hence we have
\begin{eqnarray*}
\mathbb{X}_{M_{(1,0,2)}}&=&x^{(1,0,-2)}\sum_{\underline{e}}\chi(\opname{Gr}_{\underline{e}}^{lf}(M_{(1,0,2)}))x^{Be}\\
&=&x^{(1,0,-2)}(1+2x^{(-1,1,0)}+x^{(-2,2,0)}+x^{(-2,0,2)})\\
&=&\frac{x_1^2+2x_1x_2+x_2^2+x_3^2}{x_1x_3^2}.
\end{eqnarray*}
Similarly, one can compute that 
\begin{eqnarray*}
\mathbb{X}_{M_{(1,0,0)}}&=&x^{(-1,2,0)}(1+x^{(0,-2,2)})=\frac{x_2^2+x_3^2}{x_1};\\
\mathbb{X}_{M_{(0,1,0)}}&=&x^{(0,-1,1)}(1+x^{(1,0,-1)})=\frac{x_1+x_3}{x_2};\\
\mathbb{X}_{M_{(0,0,1)}}&=&x^{(1,0,-1)}(1+x^{(-1,1,0)})=\frac{x_1+x_2}{x_3};\\
\mathbb{X}_{M_{(0,1,1)}}&=&x^{(0,-1,0)}(1+x^{(1,0,-1)}+x^{(0,1,-1)})=\frac{x_1+x_2+x_3}{x_2x_3};\\
\mathbb{X}_{M_{(1,0,1)}}&=&x^{(0,1,-1)}(1+x^{-1,1,0}+x^{(-1,-1,2)})
=\frac{x_1x_2+x_2^2+x_3^2}{x_1x_3};\\
\mathbb{X}_{M_{(1,1,1)}}&=&x^{(0,0,-1)}(1+x^{(-1,1,0)}+x^{(-1,-1,2)}+x^{(0,-1,1)})=\frac{x_1x_2+x_2^2+x_3^2+x_1x_3}{x_1x_2x_3};\\
\mathbb{X}_{M_{(1,1,0)}}&=&x^{(-1,1,0)}(1+x^{(0,-2,2)}+x^{(1,-2,1)})=\frac{x_2^2+x_3^2+x_1x_3}{x_1x_2};\\
\mathbb{X}_{M_{(1,2,0)}}&=&x^{(-1,0,0)}(1+x^{(0,-2,2)}+2x^{(1,-2,1)}+x^{(2,-2,0)})=\frac{x_1^2+x_2^2+x_3^2+2x_1x_3}{x_1x_2^2}.
\end{eqnarray*}
By Theorem~\ref{t:main-thm}, \[\mathbb{X}_{M_{(1,0,0)}},\mathbb{X}_{M_{(0,1,0)}},\mathbb{X}_{M_{(0,0,1)}},\mathbb{X}_{M_{(0,1,1)}},\mathbb{X}_{M_{(1,0,2)}},\mathbb{X}_{M_{(1,0,1)}},\mathbb{X}_{M_(1,1,1)},\mathbb{X}_{M_{(1,1,0)}},\mathbb{X}_{M_{(1,2,0)}}\] are precisely the non-initial cluster variables of $\mathcal{A}$. Moreover, for each rank vector $\underline{c}$, the denominator vector of the cluster variable $\mathbb{X}_{M_{\underline{c}}}$ is exactly $\underline{c}$.

\subsection*{Acknowledgments}
We are very grateful to the anonymous referee for significant comments and corrections they proposed.


\begin{thebibliography}{99}

\bibitem{AIR}
T. Adachi, O. Iyama and I. Reiten, \emph{$\tau$-tilting theory}, Compos. Math. \textbf{150}(2014), no. 3, 415-452.


\bibitem{Amiot}
C. Amiot, \emph{Cluster categories for algebras of global dimension 2 and quivers with potential}, Annales de l'institut Fourier \textbf{59} (2009), no. 6, 2525-2590.


\bibitem{Barot-Kussin-Lenzing}
M. Barot, D. Kussin and H. Lenzing, \emph{The Grothendieck group of a cluster category}, J. Pure Appl. Algebra \textbf{212}(1)(2008), 33-46.

\bibitem{BMRRT}
A. B. Buan, R. Marsh, M. Reineke, I. Reiten and G. Todorov, \emph{Tilting theory and cluster combinatorics}, Adv. Math. \textbf{204}(2)(2006), 572-618.




\bibitem{BuanMarshVatne}
A. B. Buan, R. Marsh and D. F. Vatne, \emph{Cluster structures from $2$-Calabi-Yau categories with loops}, Math. Z. \textbf{256}(4) (2010), 951-970.



\bibitem{ButlerRingel}
M. C. R. Butler and C. M. Ringel,\emph{Auslander-Reiten sequences with few middle terms and applications to string algebras}, Communications in Algebra \textbf{15}(1987), 145-179.

\bibitem{CalderoChapoton}
P. Caldero and F. Chapoton, \emph{Cluster algebras as Hall algebras of quiver representations}, Comment. Math. Helv. \textbf{81}(2006), no. 3, 595-616.

\bibitem{CHL17}
P. Cao, M. Huang and F. Li, \emph{Categorification of sign-skew-symmetric cluster algebras and some conjectures on g-vectors}, arXiv:1704.07549.

\bibitem{ChangZhangZhu}
W. Chang, J. Zhang and B. Zhu, \emph{On support $\tau$-tilting modules over endomorphism algebras of rigid objects}, Acta Math. Sinica, English Series \textbf{31}(9)(2015), 1508-1516.


\bibitem{Demonet}
L. Demonet, \emph{Categorification of skew-symmetrizable cluster algebras}, Algebra Represent. Theory \textbf{14}(6)(2011), 1087-1162.



\bibitem{DG}
S. Dom\'{i}nguez and C. Gei\ss, \emph{A Caldero-Chapoton formula for generalized cluster categories}, J. Algebra \textbf{399}(2014), 887-893.

\bibitem{FominZelevinsky02}
S. Fomin and A. Zelevinsky, \emph{Cluster algebras I: Foundations}, J. Amer. Math. Soc. \textbf{15}(2)(2002), 497-529.
\bibitem{FominZelevinsky03}
S. Fomin and A. Zelevinsky, \emph{Cluster algebras II: finite type classification}, Invent. Math. \textbf{154}(2003), 63-121.

\bibitem{FGL1}
C. Fu, S. Geng and P. Liu, \emph{Cluster algebras arising from cluster tubes I: integer vectors}, arXiv:1801.00709.

\bibitem{GLS08}
C. Gei\ss, B. Leclerc and J. Schr\"{o}er, \emph{Rigid modules over preprojective algebras}, Invent. Math. \textbf{165}(3)(2006), 589-632.

\bibitem{GLS14}
C. Gei\ss, B. Leclerc and J. Schr\"{o}er,
\emph{Quivers with relations for symmetrizable Cartan matrices I : Foundations}, Invent. Math. \textbf{209} (2017), 61-158.

\bibitem{GLS17}
C. Gei\ss, B. Leclerc and J. Schr\"{o}er,
\emph{Quivers with relations for symmetrizable Cartan matrices V: The Caldero-Chapoton formula}, Proc. London Math. Soc., https://doi.org/10.1112/plms.12146.

\bibitem{Ginzburg}
V. Ginzburg, \emph{Calabi-Yau algebra}, arXiv:math.AG/0612139.

\bibitem{IY}
O. Iyama and Y. Yoshino, \emph{Mutations in triangulated categories and rigid Cohen-Macaulay modules},  Invent. Math. \textbf{172}(2008), 117-168. 

\bibitem{Keller05}
B. Keller, \emph{On triangulated orbit categories}, Doc. Math. \textbf{10}(2005), 551-581.
\bibitem{Keller}
B. Keller, \emph{Deformed Calabi-Yau completions. With an appendix by Michel Van den Bergh}, J. Reine Angew. Math. \textbf{654}(2011), 125-180. 

\bibitem{Keller10}
B. Keller, \emph{Cluster algebras, quiver representations and triangulated categories}, LMS Lecture Notes Ser., \textbf{375} (2010), Cambridge, 76-160.


\bibitem{KR}
B. Keller and I. Reiten, \emph{Cluster-tilted algebras are Gorenstein and stably Calabi-Yau}, Adv. Math. \textbf{211}(1)(2007), 123-151.

\bibitem{KZ}
 S. Koenig and B. Zhu, \emph{From triangulated categories to abelian categories– cluster tilting in a general framework},  Math. Zeit. \textbf{258}(2008), 143-160.

\bibitem{LiuXie}
 P. Liu, Y. Xie,  On the relation between maximal rigid objects and $\tau$-tilting modules. Colloq. Math. \textbf{142}(2)(2016), 169-178.

\bibitem{Palu}
Y. Palu, \emph{Cluster characters for $2$-Calabi-Yau triangulated categories}, Ann. Inst. Fourier, Grenoble \textbf{56}(6)(2008), 2221-2248.

\bibitem{Plamondon11}
P. Plamondon, \emph{Cluster algebras via cluster categories with infinite-dimensional morphism spaces}, Composito Math. \textbf{147}(2011),1921-1954.


\bibitem{R11}
D. Ruppel, \emph{On quantum analogue of the Caldero-Chapoton formula}, Int. Math. Res. Not. \textbf{14} (2011), 3207-3236.
\bibitem{R15}
D. Ruppel, \emph{Quantum cluster characters for valued quivers}, Trans. Amer. Math. Soc. \textbf{367} (2015), 7061-7102.

\bibitem{Vatne11}
D. F. Vatne, \emph{Endomorphism rings of maximal rigid objects in cluster tubes}, Colloq. Math.\textbf{123}(2011), 63-93.

\bibitem{Yang12}
D. Yang, \emph{Endomorphism algebras of maximal rigid objects in cluster tubes}, Comm. Algebra \textbf{40}(2012), 4347-4371.

\bibitem{ZhouZhu}
Y. Zhou and B. Zhu, \emph{Maximal rigid subcategories in 2-Calabi-Yau triangulated
categories},  J. Algebra \textbf{348}(2011), 49-60.
\bibitem{ZhouZhu14}
Y. Zhou and B. Zhu, \emph{Cluster algebras arising from cluster tubes}, J. London Math. Soc. \textbf{89}(3)(2014), 703-723.
\end{thebibliography}
\end{document}